\documentclass[11pt,a4paper]{article}
\pdfoutput=1
\usepackage{hyperref}
\hypersetup{hypertexnames = false, bookmarksdepth = 2, bookmarksopen = true, colorlinks, linkcolor = black, citecolor = black, urlcolor = black, pdfstartview={XYZ null null 1}}

\usepackage{amsfonts}
\usepackage[fleqn, leqno]{amsmath}
\usepackage{amsthm}

\usepackage{fullpage}

\usepackage[capitalise]{cleveref}

\usepackage[utf8]{inputenc}
\usepackage[backend=biber, maxbibnames=99, datamodel=mrnumber, sortcites]{biblatex}

\usepackage{booktabs}
\usepackage{colonequals}
\usepackage{diagbox}
\usepackage{enumitem}
\usepackage{mathtools}
\usepackage{parskip}
\usepackage{thmtools}
\usepackage{tikz-cd}
\usepackage[colorinlistoftodos, textsize = footnotesize]{todonotes}
\usepackage{xparse}
\usepackage{xspace}

\usepackage[T1]{fontenc}
\usepackage{libertine}
\usepackage[libertine]{newtxmath}
\usepackage[scaled=0.83]{beramono}
\usepackage{eucal}
\usepackage{microtype}
\usepackage{gitinfo2}

\DeclareFieldFormat{mrnumber}{%
  MR\addcolon\space
  \ifhyperref
    {\href{http://www.ams.org/mathscinet-getitem?mr=1#1}{\nolinkurl{#1}}}
    {\nolinkurl{#1}}}

\renewbibmacro*{doi+eprint+url}{%
  \iftoggle{bbx:doi}
    {\printfield{doi}}
    {}%
  \newunit\newblock
  \printfield{mrnumber}%
  \newunit\newblock
  \iftoggle{bbx:eprint}
    {\usebibmacro{eprint}}
    {}%
  \newunit\newblock
  \iftoggle{bbx:url}
    {\usebibmacro{url+urldate}}
    {}}

\newcounter{todocounter}
\DeclareDocumentCommand\addreference{g}{\stepcounter{todocounter}\todo[color = blue!30]{\thetodocounter. Add reference\IfNoValueF{#1}{: #1}}\xspace}
\DeclareDocumentCommand\checkthis{g}{\stepcounter{todocounter}\todo[color = red!50]{\thetodocounter. Check this\IfNoValueF{#1}{: #1}}\xspace}
\DeclareDocumentCommand\fixthis{g}{\stepcounter{todocounter}\todo[color = orange!50]{\thetodocounter. Fix this\IfNoValueF{#1}{: #1}}\xspace}
\DeclareDocumentCommand\expand{g}{\stepcounter{todocounter}\todo[color = green!50]{\thetodocounter. Expand\IfNoValueF{#1}{: #1}}\xspace}

\declaretheoremstyle[
  spaceabove = 3pt,
  spacebelow = 3pt,
  bodyfont = \itshape,
]{first}
\declaretheoremstyle[
  spaceabove = 3pt,
  spacebelow = 3pt,
]{second}
\declaretheorem[numberwithin=section, style=first]{theorem}

\declaretheorem[sibling=theorem, style=first]{corollary}
\declaretheorem[sibling=theorem, style=first]{lemma}
\declaretheorem[sibling=theorem, style=first]{proposition}

\declaretheorem[sibling=theorem, style=second]{example}
\declaretheorem[sibling=theorem, style=second]{remark}
\declaretheorem[sibling=theorem, style=second]{definition}

\Crefname{assumption}{Assumption}{Assumptions}
\Crefname{convention}{Convention}{Conventions}
\Crefname{setup}{Setup}{Setups}

\declaretheorem[numberwithin=section, style=first, title=Theorem]{alphatheorem}
\declaretheorem[sibling=alphatheorem, style=first, title=Conjecture]{alphaconjecture}
\declaretheorem[sibling=alphatheorem, style=first, title=Corollary]{alphacorollary}

\crefname{alphatheorem}{Theorem}{Theorems}
\crefname{alphaconjecture}{Conjecture}{Conjectures}
\crefname{alphacorollary}{Corollary}{Corollaries}
\crefname{alphaproposition}{Proposition}{Propositions}

\DeclareDocumentCommand\note{g}{\todo[color = green!50]{#1}\xspace}

\mathchardef\mhyphen="2D
\newcommand\dash{\nobreakdash-\hspace{0pt}}

\DeclareMathOperator\characters{X}
\DeclareMathOperator\derived{\mathbf{D}}
\DeclareMathOperator\End{End}
\DeclareMathOperator\Hom{Hom}
\DeclareMathOperator\Ext{Ext}

\DeclareMathOperator\GL{GL}
\DeclareMathOperator\HH{H}
\DeclareMathOperator\hochschild{HH}
\DeclareMathOperator\Mat{Mat}
\DeclareMathOperator\Pic{Pic}
\DeclareMathOperator\Hilb{Hilb}
\DeclareMathOperator\rk{rk}
\DeclareMathOperator\dimvect{\underline{\operatorname{dim}}}
\DeclareMathOperator\sheafHom{\mathcal{H}om}
\DeclareMathOperator\RsheafHom{\mathbf{R}\mathcal{H}om}
\DeclareMathOperator\weight{wt}

\newcommand\cocharacters{\ensuremath{\mathrm{X}_*}}
\newcommand\bounded{\ensuremath{\mathrm{b}}}
\newcommand\Gm{\ensuremath{\mathbf{G}_{\mathrm{m}}}}
\newcommand\can{\ensuremath{\mathrm{can}}}
\newcommand\HN{\ensuremath{\mathrm{HN}}}
\newcommand\normal{\ensuremath{\mathrm{N}}}
\newcommand\tangent{\ensuremath{\mathrm{T}}}

\newcommand\tuple[1]{\ensuremath{\mathbf{#1}}}

\DeclareMathOperator\Rep{R}
\DeclareMathOperator\PG{PG}
\DeclareDocumentCommand\modulistack{om}{\IfNoValueTF{#1}{\mathcal{M}{(#2)}}{\mathcal{M}^{#1}(#2)}}
\DeclareDocumentCommand\modulispace{om}{\IfNoValueTF{#1}{\mathrm{M}{(#2)}}{\mathrm{M}^{#1}(#2)}}
\DeclareDocumentCommand\repspace{om}{\IfNoValueTF{#1}{\mathrm{R}{(#2)}}{\mathrm{R}^{#1}(#2)}}
\newcommand\moduli{\ensuremath{\mathrm{M}}}

\DeclareMathOperator\source{s}
\DeclareMathOperator\target{t}

\newcommand\semistable{\ensuremath{\mhyphen\mathrm{sst}}}
\newcommand\stable{\ensuremath{\mhyphen\mathrm{st}}}
\newcommand\semisimple{\ensuremath{\mathrm{ssimp}}}

\newcommand\field{\mathbf{k}}
\DeclareMathOperator\Spec{Spec}

\newcommand\blade[1][]{\Sigma_{#1}}
\newcommand\limitset[1][\lambda]{Z_{#1}}
\newcommand\diff{\mathrm{d} \kern 0pt}
\newcommand\group[1][\tuple{d}]{\mathrm{G}_{#1}}

\newif\ifcode

\title{Rigidity and Schofield's partial tilting conjecture \\ for quiver moduli}
\author{Pieter Belmans \and Ana-Maria Brecan \and Hans Franzen \and Gianni Petrella \and Markus Reineke}

\begin{document}
\maketitle


\begin{abstract}
  We explain how Teleman quantization can be applied
  to moduli spaces of quiver representations,
  in order to compute the higher cohomology of the endomorphism bundle of the universal bundle.
  We use this
  to prove Schofield's partial tilting conjecture,
  and to show that moduli spaces of quiver representations are (infinitesimally) rigid as varieties.
\end{abstract}

\tableofcontents

\section{Introduction}
\label{section:introduction}
There exists a rich moduli theory for representations of a quiver~$Q=(Q_0,Q_1)$,
surveyed in~\cite{MR2484736}.
Following King~\cite{MR1315461},
one fixes a dimension vector~$\tuple{d}\in\mathbb{N}^{Q_0}$
and a stability parameter~$\theta\in\Hom(\mathbb{Z}^{Q_0},\mathbb{Z})$,
and constructs a moduli space~$\moduli^{\theta\stable}(Q,\tuple{d})$ as a GIT quotient,
parametrizing isomorphism classes of~$\theta$-stable representations
of dimension vector~$\tuple{d}$.

When~$\tuple{d}$ and~$\theta$ are chosen appropriately,
the moduli space~$\moduli^{\theta\stable}(Q,\tuple{d})$
is a \emph{fine} moduli space,
parametrizing isomorphism classes of~$\theta$-(semi)stable representations.
It comes with a universal bundle~$\mathcal{U}$,
whose fiber in a point~$[V]\in\moduli^{\theta\stable}(Q,\tuple{d})$
is the representation~$V$.
The universal bundle~$\mathcal{U}$ decomposes as a direct sum~$\bigoplus_{i\in Q_0}\mathcal{U}_i$.
These constructions and results will be recalled in \cref{section:moduli}.

In this paper
we are interested in~$\sheafHom(\mathcal{U},\mathcal{U})\cong\mathcal{U}^\vee\otimes\mathcal{U}\cong\bigoplus_{i,j\in Q_0}\mathcal{U}_i^\vee\otimes\mathcal{U}_j$,
and we prove the following result on the summands,
which can be equivalently stated by saying that~$\sheafHom(\mathcal{U},\mathcal{U})$
has no higher cohomology.
The condition of being strongly amply~$\theta$-stable
is given in \cref{definition:strong-ample-stability}.

\begin{alphatheorem}
  \label{theorem:cohomology-vanishing}
  Let~$Q$ be a quiver,
  let~$\tuple{d}$ be a dimension vector,
  and let~$\theta$ be a stability parameter
  such that~$\tuple{d}$ is~$\theta$\dash coprime.
  Assume furthermore that~$\tuple{d}$ is strongly amply~$\theta$\dash stable.
  Then, for all~$i,j \in Q_0$ we have
  \begin{equation}
    \label{equation:cohomology-vanishing}
    \HH^{\geq1}(\moduli^{\theta\stable}(Q,\tuple{d}), \mathcal{U}^\vee_i \otimes \mathcal{U}_j) = 0,
  \end{equation}
  where the $\mathcal{U}_i$ are the summands of
  the universal representation on the moduli space.
\end{alphatheorem}

For this result we do not have to assume that~$Q$ is acyclic.

We will prove this using \emph{Teleman quantization},
a tool that allows to compute sheaf cohomology on a GIT~quotient
using sheaf cohomology on the quotient stack which includes the unstable locus~\cite{MR1792291}.
To do so, we recall in \cref{section:moduli}
how~$\moduli^{\theta\stable}(Q,\tuple{d})$ is constructed as a quotient
of the Zariski-open~$\repspace[\theta\stable]{Q,\tuple{d}}\subset\repspace{Q,\tuple{d}}$,
and how the bundles~$\mathcal{U}_i^\vee\otimes\mathcal{U}_j$ are the descent
of bundles which are defined on the entire~$\repspace{Q,\tuple{d}}$.
Thus we are in a position to check the conditions for Teleman quantization.

These conditions are checked in \cref{section:teleman-for-quiver-moduli}.
We determine the width of the windows in \cref{subsection:width},
and we determine the weights of the bundles~$\mathcal{U}_i^\vee\otimes\mathcal{U}_j$ in \cref{subsection:weights-endomorphism-bundle}.
The proof of \cref{theorem:cohomology-vanishing}
thus reduces to finding appropriate conditions on~$Q$, $\tuple{d}$ and~$\theta$
for which the inequality~\eqref{equation:teleman-inequality} in the statement of Teleman quantization holds,
which is done in \cref{subsection:cohomology-vanishing}.
We illustrate the setup for Teleman quantization in \cref{example:kronecker-moduli-space}.

We will now discuss three applications of this cohomology vanishing.

\paragraph{Schofield's conjecture}
The following conjecture %
is attributed to Schofield~\cite[page~80]{MR1428456}.
\begin{alphaconjecture}[Schofield]
  \label{conjecture:schofield}
  Let~$Q$ be an acyclic quiver,
  and~$\tuple{d}$ a dimension vector.
  Let~$\theta$ be a stability parameter
  such that~$\modulispace[\theta\stable]{Q,\tuple{d}}$
  is a smooth projective variety.
  Then the universal representation~$\mathcal{U}$
  is a partial tilting bundle on~$\modulispace[\theta\stable]{Q,\tuple{d}}$,
  i.e.,
  \begin{equation}
    \Ext_{\modulispace[\theta\stable]{Q,\tuple{d}}}^{\geq 1}(\mathcal{U},\mathcal{U})=0.
  \end{equation}
\end{alphaconjecture}
There is a second part to the conjecture,
which states that
this partial tilting bundle can be completed to a tilting bundle.
We will not address the second part in this paper.

The result in \cref{theorem:cohomology-vanishing} thus settles Schofield's conjecture for a large class of quiver moduli.
\begin{alphacorollary}
  \label{corollary:schofield}
  Let~$Q$ be an acyclic quiver,
  let~$\tuple{d}$ be a dimension vector,
  and let~$\theta$ be a stability parameter,
  such that~$\tuple{d}$ is~$\theta$-coprime.
  Assume furthermore that $\tuple{d}$ is strongly amply~$\theta$-stable.
  Then~$\mathcal{U}$ is a partial tilting bundle.
\end{alphacorollary}

This makes it an interesting question to compute the endomorphism algebra~$\End_{\modulispace[\theta\stable]{Q,\tuple{d}}}(\mathcal{U},\mathcal{U})$,
and thus understand whether the functor
\begin{equation}
  \label{equation:fully-faithful-functor}
  \Phi_{\mathcal{U}}\colon
  \derived^\bounded(\field Q)\to\derived^\bounded(\modulispace[\theta\stable]{Q,\tuple{d}}):
  V\mapsto\RsheafHom_{\mathcal{O}_XQ}(\mathcal{U},-\otimes\mathcal{O}_{\modulispace[\theta\stable]{Q,\tuple{d}}})
\end{equation}
given by the universal representation is fully faithful,
i.e.,
whether
\begin{equation}
  \label{equation:End-is-kQ}
  \End_{\modulispace[\theta\stable]{Q,\tuple{d}}}(\mathcal{U},\mathcal{U})
  \cong
  \field Q
\end{equation}
by the morphism naturally induced by~$\Phi_{\mathcal{U}}$.
In the thin case and with respect to the canonical stability condition
the isomorphism \eqref{equation:End-is-kQ} was established by Altmann--Hille in~\cite[Theorem~1.3]{MR1688469}.
Another known case is that of quiver flag varieties,
which are quiver moduli for acyclic quivers with a unique source~$i_0$
and dimension vector~$\tuple{d}$ such that~$d_{i_0}=1$.
In \cite[Theorem~1.2]{MR2772068} Craw shows that both parts of Schofield's conjecture hold,
i.e.,~$\mathcal{U}$ is partial tilting and can be extended to a tilting bundle,
whilst Craw--Ito--Karmazyn show in \cite[Example~2.9]{MR3803802}
that \eqref{equation:End-is-kQ} holds in this setting.

The isomorphism \eqref{equation:End-is-kQ}
for a much larger class of quiver moduli
is studied by a subset of the authors in~\cite{vector-fields},
building upon the results of this paper.
The main result in op.~cit.,
which uses methods very different from the ones used in this paper,
is that,
under the appropriate conditions,
there exists explicit isomorphisms
\begin{equation}
  \Hom_{\modulispace[\theta\stable]{Q,\tuple{d}}}(\mathcal{U}_i,\mathcal{U}_j)
  \cong
  \HH^0(\modulispace[\theta\stable]{Q,\tuple{d}},\mathcal{U}_i^\vee\otimes\mathcal{U}_j)
  \cong
  e_j\field Qe_i,
\end{equation}
for all~$i,j\in Q_0$.
This result thus naturally complements \cref{theorem:cohomology-vanishing}.

In fact, the fully faithfulness in~\eqref{equation:fully-faithful-functor}
has a counterpart for moduli of vector bundles on curves.
There exists a rich dictionary between
results for moduli spaces of quiver representations,
and for moduli spaces of vector bundles on curves,
see, e.g.,~\cite{MR3882963} for a survey on some aspects of this dictionary.
Taking this dictionary for granted,
one can take a result for one type of moduli spaces
and try to obtain the analogous result for the other type.

For moduli spaces of vector bundles
the analogue of~\eqref{equation:fully-faithful-functor}
is the Fourier--Mukai functor
\begin{equation}
  \Phi_{\mathcal{E}}\colon\derived^\bounded(C)\to\derived^\bounded(\mathrm{M}_C(r,\mathcal{L})),
\end{equation}
where~$\mathcal{E}$ is the universal vector bundle on~$C\times\mathrm{M}_C(r,\mathcal{L})$,
and~$\mathrm{M}_C(r,\mathcal{L})$ is the moduli space of
stable vector bundles of rank~$r$ and determinant~$\mathcal{L}$
such that~$\gcd(r,\deg\mathcal{L})=1$
and~$C$ is a smooth projective curve of genus~$g\geq 2$.
Its fully faithfulness is shown in various levels of generality in~\cite{MR3764066,MR3954042,MR3713871,MR4651618}.

\paragraph{Rigidity}
The next application is inspired by the same curve-quiver dictionary.
Recall that in the seminal paper~\cite{MR0384797} Narasimhan--Ramanan show
that the deformation theory of the curve
is the same as that of the moduli space~$\mathrm{M}_C(r,\mathcal{L})$,
again under the assumption that~$\gcd(r,\deg\mathcal{L})=1$ to ensure the moduli space is smooth and projective.
More precisely,
they establish
\begin{equation}
  \label{equation:narasimhan-ramanan}
  \HH^i(\mathrm{M}_C(r,\mathcal{L}),\tangent_{\mathrm{M}_C(r,\mathcal{L})})
  \cong
  \HH^i(C,\tangent_C)
\end{equation}
which is well-known to be~$3g-3$-dimensional for~$i=1$,
and vanishes for~$i\neq 1$.
For quiver representations and their moduli,
the deformation theory of the quiver
is to be interpreted as the deformation theory of the path algebra of the quiver,
and thus we are interested in its (second) Hochschild cohomology.
By~\cite[\S1.6]{MR1035222} we have that
\begin{equation}
  \label{equation:happel}
  \hochschild^2(\field Q)=0,
\end{equation}
if~$Q$ is acyclic,
thus the path algebra is \emph{rigid} as an associative algebra.
Under our dictionary between quivers and curves
the analogue of~\eqref{equation:narasimhan-ramanan}
becomes an isomorphism
\begin{equation}
  \HH^1(\modulispace[\theta\stable]{Q,\tuple{d}},\tangent_{\modulispace[\theta\stable]{Q,\tuple{d}}})
  \cong
  \hochschild^2(\field Q),
\end{equation}
and thus by the vanishing in~\eqref{equation:happel}
that~$\modulispace[\theta\stable]{Q,\tuple{d}}$
is \emph{rigid as a variety}.
This is proved in \cref{subsection:rigidity}.
The global sections of~$\tangent_{\modulispace[\theta\stable]{Q,\tuple{d}}}$
should be related to~$\hochschild^1(\field Q)$,
which is studied in~\cite{vector-fields}.

Because~\eqref{equation:happel} actually reads~$\hochschild^{\geq 2}(\field Q)$
we in fact expect higher cohomology vanishing for the tangent bundle,
We can prove the rigidity and further vanishing using \cref{theorem:cohomology-vanishing}
and the 4-term sequence~\eqref{equation:4-term-exact-sequence}.

\begin{alphacorollary}
  \label{corollary:rigidity}
  Let~$Q$ be an acyclic quiver,
  let~$\tuple{d}$ be a dimension vector,
  and let~$\theta$ be a stability parameter,
  such that~$\tuple{d}$ is~$\theta$-coprime.
  Assume furthermore that $\tuple{d}$ is strongly amply~$\theta$-stable.
  Then~$\modulispace[\theta\stable]{Q,\tuple{d}}$ satisfies
  \begin{equation}
    \HH^{\geq 1}(\modulispace[\theta\stable]{Q,\tuple{d}},\tangent_{\modulispace[\theta\stable]{Q,\tuple{d}}})=0,
  \end{equation}
  so in particular it is (infinitesimally) rigid,
  i.e.,~$\HH^1(\modulispace[\theta\stable]{Q,\tuple{d}},\tangent_{\modulispace[\theta\stable]{Q,\tuple{d}}})=0$.
\end{alphacorollary}

Thus, it is not possible to deform a quiver moduli space,
or put more colloquially: ``quiver moduli have no moduli''.
The only cases where this rigidity was known
were situations in which an explicit description of the moduli space was known
(e.g., for Kronecker or subspace quivers),
or in the case of Fano toric quiver moduli,
by applying Danilov–Steenbrink–Bott vanishing,
see, e.g.,~\cite[Theorem~2.4(i)]{MR1916637}.

In~\cite{2303.08522v2} Domokos proves a related result,
which states that when one fixes the dimension
there exist only \emph{finitely} many quiver moduli of that dimension,
where a priori the number could also be countable.
There are three variations on the finiteness statement in op.~cit.,
with different conditions on the dimension vectors that one considers.

The combination of \cref{corollary:rigidity} and Domokos' finiteness result
explains how the classification of quiver moduli in a given dimension,
especially when they are Fano,
is a particularly interesting question.

\smallskip

The computation of the global sections of the tangent bundle,
i.e., the vector fields of the moduli space,
is addressed in~\cite{vector-fields}.
It is shown that,
under favourable circumstances,
there exists an isomorphism~$\hochschild^1(\field Q)\cong\HH^0(\modulispace[\theta\stable]{Q,\tuple{d}},\tangent_{\modulispace[\theta\stable]{Q,\tuple{d}}})$,
as predicted by the dictionary between quivers and curves.

\paragraph{Height-zero moduli spaces}
The final application involves moduli spaces of sheaves on~$\mathbb{P}^2$,
and in particular those of height zero,
as introduced by Drezet~\cite{MR0916199,MR0944602,MR0915184}.
These are moduli spaces~$\moduli_{\mathbb{P}^2}(r,\mathrm{c}_1,\mathrm{c}_2)$
of (semi)stable sheaves on~$\mathbb{P}^2$
which can be characterized by having~%
$\Pic\moduli_{\mathbb{P}^2}(r,\mathrm{c}_1,\mathrm{c}_2)\cong\mathbb{Z}$~\cite[Th\'eor\`eme~2]{MR0915184}.
The precise definition is recalled in \cref{subsection:height-zero}.

Using the parameters~$(r,\mathrm{c}_1,\mathrm{c}_2)=(1,0,n)$
we describe the Hilbert scheme of~$n$ points on~$\mathbb{P}^2$
as a fine moduli space of stable sheaves.
In~\cite[Theorem~1.2]{MR3397451} it is shown that the Fourier--Mukai functor~%
$\Phi_{\mathcal{I}}\colon\derived^\bounded(\mathbb{P}^2)\to\derived^\bounded(\Hilb^n\mathbb{P}^2)$,
where~$\mathcal{I}$ is the universal ideal sheaf,
is fully faithful for all~$n\geq 2$.
In~\cite[Proposition~29]{MR3950704} it is shown that this implies that~%
$\HH^1(\Hilb^n\mathbb{P}^2,\tangent_{\Hilb^n\mathbb{P}^2})\cong\hochschild^2(\mathbb{P}^2)$ is non-zero.
As mentioned in Remark~30 of op.~cit.,
the same method to compute the infinitesimal deformations
works verbatim for other smooth projective moduli spaces,
provided the Fourier--Mukai functor is fully faithful.

This gives us examples where the functor cannot be fully faithful,
using the correspondence between moduli spaces of height zero
and certain Kronecker moduli~\cite[Th\'eor\`eme~2]{MR0916199}.
We prove the following in \cref{subsection:height-zero}.

\begin{alphacorollary}
  \label{corollary:height-zero}
  Let~$\moduli_{\mathbb{P}^2}(r,\mathrm{c}_1,\mathrm{c}_2)$ be a smooth projective
  fine moduli space of stable sheaves on~$\mathbb{P}^2$,
  with universal object~$\mathcal{E}$.
  If it is of height zero,
  then
  \begin{equation}
    \label{equation:height-zero-functor}
    \Phi_{\mathcal{E}}\colon\derived^\bounded(\mathbb{P}^2)\to\derived^\bounded(\moduli_{\mathbb{P}^2}(r,\mathrm{c}_1,\mathrm{c}_2))
  \end{equation}
  is not fully faithful.
\end{alphacorollary}

The special case of the corollary where~$(r,\mathrm{c}_1,\mathrm{c}_2)=(4,1,3)$
was established in~\cite{MR4223526}.

It is also possible to obtain the result in \cref{corollary:height-zero} using other methods,
as explained in \cref{remark:associated-exceptional-bundle},
with ingredients from the original papers by Drezet.

\paragraph{Acknowledgements}

The first author would like to thank Ben Gould, Dmitrii Pedchenko and Fabian Reede
for conversations about moduli spaces of height zero
related to \cref{remark:associated-exceptional-bundle}.
We want to thank Alastair Craw for interesting discussions regarding the case of quiver flag varieties.

P.B. was partially supported by the Luxembourg National Research Fund (FNR--17113194). \\
H.F. was partially supported by the Deutsche Forschungsgemeinschaft (DFG, German Research Foundation) SFB-TRR~358/1~2023 ``Integral Structures in Geometry and Representation Theory'' (491392403). \\
G.P. was supported by the Luxembourg National Research Fund (FNR--17953441). \\
M.R. was supported by the Deutsche Forschungsgemeinschaft (DFG, German Research Foundation) CRC-TRR~191 ``Symplectic structures in geometry, algebra and dynamics'' (281071066)

\section{Moduli spaces of quiver representations}
\label{section:moduli}

We will first set up the notation,
and recall the construction of various moduli spaces of quiver representations
with their universal objects.
Throughout we will let~$\field$ be an algebraically closed field of characteristic~0.
Initially this is not necessary, but it will be for the main results.

A \emph{quiver} $Q$ is a finite directed graph,
with set of vertices $Q_0$ and set of arrows $Q_1$,
together with two maps $\source, \target\colon Q_1 \to Q_0$,
which assign to an arrow $a \in Q_1$ its source and target, respectively.

A \emph{representation} $M$ of $Q$ over a field $\field$
is the data of a finite-dimensional vector space~$M_i$
for each vertex~$i \in Q_0$
together with a linear map~$M_a\colon M_{\source(a)} \to M_{\target(a)}$
for each arrow~$a \in Q_1$.
For a representation~$M$ of~$Q$,
we define the \emph{dimension vector} $\dimvect M \in \smash{\mathbb{N}^{Q_0}}$ as the tuple $(\dim M_i)_{i \in Q_0}$.

We will briefly recall the construction of the moduli stacks and spaces,
working over the field~$\field$.
For a more formal definition,
valid over arbitrary bases,
the interested reader can consult~\cite[\S3]{2210.00033v1}.
Given a fixed dimension vector $\tuple{d} \in \smash{\mathbb{N}_0^{Q_0}}$,
where we write~$\tuple{d}=(d_i)_{i\in Q_0}$,
we define the \emph{representation variety}
as the affine space
\begin{equation}
  \label{equation:representation-variety}
  \repspace{Q,\tuple{d}}\colonequals \prod_{a \in Q_1} \Mat_{d_{\target(a)} \times d_{\source(a)}}(\field).
\end{equation}
A point of $\repspace{Q,\tuple{d}}$ is the same as a representation of $Q$ on the collection of vector spaces $\field^{d_i}$.
Let the reductive linear algebraic group $\group=\group(\field)$ be defined by
\begin{equation}
  \group(\field)\colonequals \prod_{i \in Q_0} \GL_{d_i}(\field).
\end{equation}
This group acts on the left on $\repspace{Q,\tuple{d}}$ via change of basis;
more precisely $g = (g_i)_{i \in Q_0}\in\group$ acts on the point~$M = (M_a)_{a \in Q_1}$ by
\begin{equation}
  g\cdot M \colonequals (g_{\target(a)}M_a\,g_{\source(a)}^{-1})_{a \in Q_1}.
\end{equation}
Two elements of $\repspace{Q,\tuple{d}}$ lie in the same $\group$-orbit
if and only if the corresponding representations of $Q$ are isomorphic.
Consider the closed central subgroup $\Delta = \{(z\cdot \mathrm{I}_{d_i})_{i \in Q_0} \mid z \in \Gm\}$ of $\group$.
Its elements act trivially on $\repspace{Q,\tuple{d}}$,
whence the action on $\repspace{Q,\tuple{d}}$ descends to an action of the reductive group
\begin{equation}
  \PG_{\tuple{d}} \colonequals \group/\Delta.
\end{equation}

For later use we also recall the Euler pairing of the quiver,
which for~$\alpha, \beta \in \mathbb{Z}^{Q_0}$
is defined as
\begin{equation}
  \langle \alpha, \beta \rangle = \sum_{i \in Q_0} \alpha_i \beta_i ~- \sum_{a \in Q_1} \alpha_{\source(a)} \beta_{\target(a)}.
\end{equation}
Note that the notation does not include $Q$, as it will be clear from the context.

\paragraph{Moduli of (semi)stable representations}
Let $\theta\in\Hom(\mathbb{Z}^{Q_0},\mathbb{Z})$ be a \emph{stability parameter}.
We will identify $\Hom(\mathbb{Z}^{Q_0},\mathbb{Z})$ with its dual~$\mathbb{Z}^{Q_0}$ by the dot product.
Assume that $\theta(\tuple{d}) = 0$.
We define a notion of stability with respect to $\theta$,
following King~\cite{MR1315461} (up to a different sign convention).
\begin{definition}
  \label{definition:semistability}
  A representation $M$ of $Q$ of dimension vector $\tuple{d}$
  is called $\theta$-\emph{semistable} (resp.~$\theta$\dash\emph{stable})
  if any proper non-zero subrepresentation $N$ of $M$
  satisfies the inequality~$\theta(\dimvect N) \leq 0$ (resp.~$\theta(\dimvect N) < 0$).
\end{definition}

Let
\begin{equation}
  \repspace[\theta\semistable]{Q,\tuple{d}}
  \subseteq\repspace[\theta\stable]{Q,\tuple{d}}
  \subseteq\repspace{Q,\tuple{d}}
\end{equation}
denote the Zariski open subsets of semistable and stable representations, respectively.
Note that the semistable locus may be empty,
or the stable locus may be empty while the semistable locus is not.

We will sometimes use the canonical stability parameter~$\theta_\can$ for a dimension vector~$\tuple{d}$,
which is given by the morphism~$\langle\tuple{d},-\rangle-\langle-,\tuple{d}\rangle$.

We associate the character $\chi_\theta$ of $\group$
to the stability parameter~$\theta$,
by defining
\begin{equation}
  \chi_\theta(g) = \prod_{i \in Q_0} \det(g_i)^{-\theta_i};
\end{equation}
note the minus sign in the exponent.
As $\theta(\tuple{d}) = 0$,
the character $\chi_{\theta}$ vanishes on $\Delta$,
so it is also a character of $\PG_{\tuple{d}}$.
Let $L(\theta)$ be the trivial line bundle on $\repspace{Q,\tuple{d}}$
equipped with the $\PG_\tuple{d}$-linearization given by $\chi_\theta$.
King shows in~\cite[Theorem 4.1]{MR1315461} that $\theta$-semistability
agrees with semistability with respect to $L(\theta)$,
and $\theta$-stability is the same as proper stability in the sense of Mumford's GIT.
In op.~cit., the action by~$\group$ is used,
so the points are not properly stable in the sense of Mumford,
but the only difference is a common~$\Gm$-stabilizer.

\begin{definition}
  \label{definition:moduli}
  We define the $\theta$-semistable and $\theta$-stable \emph{moduli spaces} as
  \begin{align}
    \modulispace[\theta\semistable]{Q,\tuple{d}} & \colonequals \repspace{Q,\tuple{d}}/\!\!/_{L(\theta)} \PG_{\tuple{d}}, \\
    \modulispace[\theta\stable]{Q,\tuple{d}}     & \colonequals \repspace{Q,\tuple{d}}/_{L(\theta)} \PG_{\tuple{d}}.
  \end{align}
  Moreover, we define the $\theta$-semistable and $\theta$-stable \emph{moduli stacks} as
  \begin{align}
    \modulistack[\theta\semistable]{Q,\tuple{d}} & \colonequals [\repspace[\theta\semistable]{Q,\tuple{d}}/\group], \\
    \modulistack[\theta\stable]{Q,\tuple{d}}     & \colonequals [\repspace[\theta\stable]{Q,\tuple{d}}/\group].
  \end{align}
\end{definition}

In the case $\theta = 0$, we write $\modulispace[\semisimple]{Q,\tuple{d}}$
for the moduli space of $0$-semistable representations,
as well as $\modulistack{Q,\tuple{d}} = \modulistack[0\semistable]{Q,\tuple{d}}$
for the moduli stack.
These parameterize the semisimple representations of dimension vector~$\tuple{d}$.

The relations between the affine, semistable and stable quotients are summarised in the following diagram:
\begin{equation}
  \begin{tikzcd}[row sep=small, ampersand replacement=\&]
    \repspace[\theta\stable]{Q,\tuple{d}} \arrow[r,hook] \arrow[d]     \& \repspace[\theta\semistable]{Q,\tuple{d}} \arrow[r,hook] \arrow[d]     \& \repspace{Q,\tuple{d}} \arrow[d] \\
    \modulistack[\theta\stable]{Q,\tuple{d}} \arrow[r,hook] \arrow[d] \& \modulistack[\theta\semistable]{Q,\tuple{d}} \arrow[r,hook] \arrow[d] \& \modulistack{Q,\tuple{d}} \arrow[d] \\
    \modulispace[\theta\stable]{Q,\tuple{d}} \arrow[r, hook]          \& \modulispace[\theta\semistable]{Q,\tuple{d}} \arrow[r]                \& \modulispace[\semisimple]{Q,\tuple{d}}.
    \end{tikzcd}
\end{equation}
The arrows on the first row are open immersions.
So are the leftmost arrows in the second and third row.
The morphism $\modulistack[\theta\stable]{Q,\tuple{d}} \to \modulispace[\theta\stable]{Q,\tuple{d}}$
is a $\Gm$-gerbe,
the composition $\repspace[\theta\stable]{Q,\tuple{d}} \to \modulispace[\theta\stable]{Q,\tuple{d}}$
is a principal $\PG_\tuple{d}$-bundle in the fpqc topology,
and $\modulispace[\theta\semistable]{Q,\tuple{d}} \to \modulispace[\semisimple]{Q,\tuple{d}}$ is projective.
The semisimple moduli space $\modulispace[\semisimple]{Q,\tuple{d}}$ is affine,
and if~$Q$ is acyclic
this moduli space is isomorphic to~$\Spec\field$
by the classification of simple modules.
The moduli space $\modulispace[\theta\stable]{Q,\tuple{d}}$ is smooth.

\begin{definition}
  A dimension vector $\tuple{d} \in \mathbb{N}^{Q_0}$ is called \emph{indivisible}
  if $\gcd_{i \in Q_0}(d_i) = 1$.
  Furthermore, it is called \emph{$\theta$-coprime} if,
  for any $\tuple{d'} \in \mathbb{N}^{Q_0}$ such that $\tuple{d'} \leq \tuple{d}$,
  we have $\theta(\tuple{d'}) \neq 0$,
  unless $\tuple{d'} \in \{0,\tuple{d}\}$.
\end{definition}
We see that if $\tuple{d}$ is $\theta$-coprime, then it is indivisible.
Moreover, if $\tuple{d}$ is $\theta$-coprime,
then every $\theta$-semistable representation of dimension vector $\tuple{d}$ is $\theta$-stable.
In particular, if the quiver $Q$ is acyclic and if~$\tuple{d}$ is~$\theta$\dash coprime,
the moduli space~$\modulispace[\theta\stable]{Q,\tuple{d}}$ is a smooth projective variety.

We also recall the following definition from~\cite{MR3683503}.
\begin{definition}
  \label{definition:ample-stability}
  A dimension vector~$\tuple{d}$ is said to be \emph{amply~$\theta$-stable} if
  \begin{equation}
    \operatorname{codim}_{\Rep(Q,\tuple{d})}(\Rep(Q,\tuple{d})\setminus\Rep^{\theta\stable}(Q,\tuple{d}))\geq 2.
  \end{equation}
\end{definition}
This guarantees that the Picard rank of~$\modulispace[\theta\stable]{Q,\tuple{d}}$ will be maximal,
i.e., it is equal to~$\#Q_0-1$.
We will introduce a stronger version of this condition in \cref{definition:strong-ample-stability}.

\paragraph{Universal representations}
For every $i \in Q_0$,
let $U_i$ be the trivial vector bundle of rank~$d_i$ on $\repspace{Q,\tuple{d}}$.
We equip it with an action of $\group$.
For a vector $v \in \field^{d_i} = (U_i)_M$ in the fiber over a point $M \in \repspace{Q,\tuple{d}}$,
the group element $g=(g_i)_{i\in Q_0} \in \group$ acts by
\begin{equation}
  g \cdot v = g_iv
\end{equation}
which lies in the fiber over $g\cdot M$.
For an arrow $a \in Q_1$,
let $U_a\colon U_{\source(a)} \to U_{\target(a)}$ be
the morphism which on the fibers over a point $M \in \repspace{Q,\tuple{d}}$
sends a vector $v \in (U_{\source(a)})_M$ to $M_a(v) \in (U_{\target(a)})_M$.
This morphism is clearly~$\group$\dash equivariant.

The bundles $U_i$ do not descend to $\modulispace[\theta\stable]{Q,\tuple{d}}$
because the stabilizer $\Delta$ does not act trivially on the fibers.
We need to twist by certain line bundles to achieve this.

Let~$\tuple{a} \in \mathbb{Z}^{Q_0}$.
Define~$L(\tuple{a})$ as the trivial line bundle on~$\Rep(Q,\tuple{d})$,
so that on its fibers the element~$g \in \group$ acts by
the character~$\chi_\tuple{a}(g) = \prod_{i \in Q_0} \det(g_i)^{-a_i}$.
If we assume that~$\tuple{d}$ is indivisible,
then there exists a tuple of integers~$\tuple{a} \in \mathbb{Z}^{Q_0}$
for which~$\tuple{a}\cdot\tuple{d}=\sum_{i \in Q_0}a_i d_i = 1$.

Let $U_i(\tuple{a}) \colonequals U_i \otimes L(\tuple{a})$ on~$\Rep(Q,\tuple{d})$
for~$\tuple{a}$ such that~$\tuple{a}\cdot\tuple{d}=1$.
By the choice of $\tuple{a}$, the stabilizer $\Delta$ acts trivially,
so these vector bundles, once restricted to the $\theta$-stable locus,
descend to the quotient~$\modulispace[\theta\stable]{Q,\tuple{d}}$.
The morphisms~$U_a \otimes \operatorname{id}_{L(\tuple{a})}\colon U_{\source(a)}(\tuple{a}) \to U_{\target(a)}(\tuple{a})$
descend to morphisms~$\mathcal{U}_a(\tuple{a})\colon \mathcal{U}_{\source(a)}(\tuple{a}) \to \mathcal{U}_{\target(a)}(\tuple{a})$.

\begin{definition}
  The data $\mathcal{U}(\tuple{a}) = ((\mathcal{U}_i(\tuple{a}))_{i \in Q_0},(\mathcal{U}_a(\tuple{a}))_{a \in Q_1})$
  is a representation of $Q$ in vector bundles on $\modulispace[\theta\stable]{Q,\tuple{d}}$.
  It is called a \emph{universal representation}.
\end{definition}

\begin{remark}
  \label{remark:uniqueness-up-to-twist-universal-bundle}
  Different choices of $\tuple{a}$ give rise to non-isomorphic universal representations.
  However, for~$\tuple{a}$ and~$\tuple{b} \in \mathbb{Z}^{Q_0}$ as before,
  i.e., $\tuple{a}\cdot\tuple{d}=\tuple{b}\cdot\tuple{d}=1$,
  we have universal representations~$\mathcal{U}(\tuple{a})$ and~$\mathcal{U}(\mathbf{b})$.
  The line bundle $L(\tuple{b} - \tuple{a})$
  descends to a line bundle $\mathcal{L}(\mathbf{b} - \tuple{a})$ on $\modulispace[\theta\stable]{Q,\tuple{d}}$
  because~$(\mathbf{b}-\tuple{a})\cdot\tuple{d}=0$ making the stabilizer act trivially,
  and it gives rise to the isomorphism $\mathcal{U}(\tuple{a}) \otimes \mathcal{L}(\tuple{b} - \tuple{a}) \cong \mathcal{U}(\tuple{b})$.
\end{remark}

\section{Teleman quantization for quiver moduli}
\label{section:teleman-for-quiver-moduli}
The following section is the technical heart of the paper.
We will explain how to set up Teleman quantization for quiver moduli,
by recalling the Hesselink stratification
and relating it to the Harder--Narasimhan stratification.
This allows us to compute the width of the windows in \cref{subsection:width},
and the weights of the endomorphisms of the universal representation in \cref{subsection:weights-endomorphism-bundle}.

\subsection{The Hesselink stratification and Teleman's quantization theorem}
\label{subsection:hesselink-and-teleman}
Let $G$ be a linearly reductive algebraic group,
and let $R$ be an affine variety over $\field$ on which $G$ acts.
In our application to quiver moduli we will let~$G$ be~$\group$ or~$\PG_{\tuple{d}}$,
and~$R$ the representation variety~\eqref{equation:representation-variety}.

Let~$\lambda\colon\Gm\to G$ be a 1-parameter subgroup,
and let~$\chi\colon G \to \Gm$ be a character of $G$.
We denote by~$\langle\chi,\lambda\rangle$ the integer exponent in the identity~$\chi \circ \lambda (z) = z^{\langle\chi,\lambda\rangle}$.
Recall that the Hilbert--Mumford criterion for semistability states the following.

\begin{theorem}
  \label{theorem:hilbert-mumford}
  A point $x \in R$ is $\chi$-semistable if and only if $ \langle\chi,\lambda\rangle \geq 0$
  for every one-parameter subgroup~$\lambda\colon \Gm \to G$
  for which $\lim_{z \to 0} \lambda(z)x$ exists.
\end{theorem}
Given an unstable point $x \in R$, Kempf finds in~\cite{MR0506989} a one-parameter subgroup which is
``most responsible'' for its instability.
This construction was used by Hesselink in~\cite{MR0514673}
to obtain a stratification of $R$ into locally closed subsets.
The Hesselink stratification is also the basis for Teleman's quantization theorem.
We will thus briefly recall this theory in the following.
A comprehensive reference is~\cite[\S 12 and \S 13]{MR0766741}.

Fix a maximal torus $T \subseteq G$ and let $W$ be the corresponding Weyl group.
Consider the lattice $\cocharacters(T)$ of one-parameter subgroups of $T$.
Let $(-,-)$ be a $W$-invariant inner product on the vector space $\cocharacters(T)_\mathbb{R}$
such that $(\lambda,\lambda) \in \mathbb{Z}$ for all one-parameter subgroups $\lambda$ of $T$.
The induced norm $\|-\|$ is also $W$-invariant.
Therefore, we may extend the norm to one-parameter subgroups $\lambda$ of $G$
by setting $\|\lambda\|\colonequals\|g\lambda g^{-1}\|$ where $g \in G$ is such that $g\lambda g^{-1}$ lies in $T$.

\begin{definition}
  Let $x \in R$. We define the \emph{normalized Hilbert--Mumford weight} of $x$ as
  \begin{equation}
    m(x)\colonequals \inf\left\{ \frac{\langle\chi,\lambda\rangle}{\|\lambda\|}\ \middle|\ 1 \neq \lambda \in \cocharacters(G) \text{ such that } \lim_{z\to 0} \lambda(z)x \text{ exists}\right\},
  \end{equation}
  and the set of one-parameter subgroups associated to $x$ as
  \begin{equation}
    \Lambda(x) \colonequals \left\{ \lambda \in \cocharacters(G)\ \middle|\ \lambda \text{ is primitive, } \lim_{z\to 0} \lambda(z)x \text{ exists, and } \frac{\langle\chi,\lambda\rangle}{\|\lambda\|} = m(x) \right\}.
  \end{equation}
\end{definition}

For a one-parameter subgroup $\lambda$ of $G$ we define
\begin{align}\label{definition:P_lambda}
  L_\lambda &:= \{ g \in G\mid \lambda(z)g\lambda(z)^{-1} = g \text{ for all } z \in \Gm \}, \text{ and}\\
  P_\lambda &:= \{ g \in G\mid \lim_{z \to 0} \lambda(z)g\lambda(z)^{-1} \text{ exists}\}.
\end{align}
These are closed subgroups of $G$.
Moreover, $P_\lambda$ is a parabolic subgroup with Levi factor $L_\lambda$.

We define two subsets of $R$ as follows,
where the first is the fixed locus of~$\lambda$:
\begin{align}
  R_\lambda   & \colonequals \{ x \in R \mid \lambda(z)x = x \text{ for all } z \in \Gm \}, \text{ and} \\
  R_\lambda^+ & \colonequals \{ x \in R \mid \lim_{z \to 0}\lambda(z)x \text{ exists}\}.
\end{align}
Both subsets are closed in the Zariski topology.
The group $L_\lambda$ acts on $R_\lambda$, while $P_\lambda$ acts on $R_\lambda^+$.
The morphism
\begin{equation}
  \label{equation:p-lambda}
  p_\lambda\colon R_\lambda^+ \to R_\lambda:x\mapsto\lim_{z \to 0}\lambda(z)x
\end{equation}
is $P_\lambda$-equivariant via $P_\lambda \to L_\lambda$.

\begin{definition}
  Let $\lambda$ be a one-parameter subgroup of $G$,
  and let $[\lambda]$ be its $G$-conjugacy class inside~$\cocharacters(G)$.
  We define the following subsets of $R$:
  \begin{enumerate}
    \item $S_{[\lambda]}\colonequals \{ x \in R \mid \Lambda(x) \cap [\lambda] \neq \emptyset \}$,
      which is called the \emph{Hesselink stratum} of $[\lambda]$;
    \item $\blade[\lambda]\colonequals \{ x \in R \mid \lambda \in \Lambda(x) \}\subset S_{[\lambda]}$,
      which is called the \emph{blade} of $\lambda$; and
    \item $Z_\lambda\colonequals \{ x \in R_\lambda \mid \lambda \in \Lambda(x)\} = \blade[\lambda] \cap R_\lambda\subset\blade[\lambda]$,
      which is called the \emph{limit set} of $\lambda$.
  \end{enumerate}
\end{definition}

Now we proceed to describe the Hesselink stratification.
It is an algebraic analog of the Kempf--Ness stratification.

\begin{theorem}[Hesselink]
  \label{theorem:hesselink}
  The unstable locus (or null cone) inside $R$ admits a decomposition
  \begin{equation}
    \label{equation:hesselink-stratification}
    R\setminus R^{\chi\semistable} = \bigsqcup_{[\lambda]} S_{[\lambda]},
  \end{equation}
  as a finite disjoint union into $G$-invariant locally closed subsets
  ranging over all the $G$-conjugacy classes of primitive $\lambda \in \cocharacters(G)$
  such that $\langle \chi,\lambda \rangle < 0$.
  Moreover,
  \begin{enumerate}
    \item Each limit set $Z_\lambda$ is $L_\lambda$-invariant and open inside $R_\lambda$;
    \item Each blade satisfies $\blade[\lambda] = p_\lambda^{-1}(Z_\lambda)$,
      where~$p_\lambda$ is as in~\eqref{equation:p-lambda},
      and is hence a~$P_\lambda$-invariant open subset of~$R_\lambda^+$;
    \item Each stratum $S_{[\lambda]}$ satisfies $S_{[\lambda]} = G \cdot \blade[\lambda]$; and
    \item The action map $\sigma\colon G\times \blade[\lambda] \to G\cdot \blade[\lambda] = S_{[\lambda]}$
      induces an isomorphism with the associated fiber bundle,
      i.e.,~$G\times^{P_\lambda} \blade[\lambda] \cong S_{[\lambda]}$.
  \end{enumerate}
\end{theorem}
The complement of the null cone can be included in the decomposition \eqref{equation:hesselink-stratification}
by using the trivial one-parameter subgroup.

\begin{remark}
  \label{remark:indexing-set}
  The disjoint union of~\eqref{equation:hesselink-stratification} ranges over $G$-conjugacy classes
  of primitive one-parameter subgroups $\lambda$ of $G$ such that $\langle \chi,\lambda \rangle < 0$.
  By a result of Kempf~\cite[Theorem 2.15]{MR3261979},
  every such $G$-conjugacy class contains a representative $\lambda' \in \cocharacters(T)$
  which is primitive and satisfies $\langle \chi,\lambda' \rangle < 0$.
  This representative is unique up to the action of $W$.
\end{remark}

Now we state the result which is most important for our purposes -- Teleman's quantization theorem.
We do not state its most general version, but only one which will be sufficient for our purposes.
This result was given in great generality by Halpern-Leistner in~\cite{MR3327537}.

\begin{theorem}[Teleman]
  \label{theorem:teleman-quantization}
  Let~$G$ be a reductive algebraic group
  acting on an affine variety~$R$.
  Let $\{S_{[\lambda]}\}$ be the Hesselink stratification~\eqref{equation:hesselink-stratification} of the unstable locus
  with respect to the character $\chi$ from \cref{theorem:hesselink}.
  Assume that all limit sets $Z_\lambda$ are connected.
  For each $\lambda$, define $\eta_\lambda$ as the weight of the action of $\lambda$
  on the determinant of the conormal bundle restricted to the limit set, i.e.,
  \begin{equation}
    \label{equation:eta-lambda}
    \eta_\lambda\colonequals\weight_\lambda \left( (\det \normal_{S_{[\lambda]}/R}^\vee) |_{Z_\lambda} \right).
  \end{equation}
  Let $\mathcal{F}$ be a $G$-linearized vector bundle on $R$.
  If, for each $\lambda$, the weights of the action of $\lambda$ on $\mathcal{F}|_{Z_\lambda}$
  are strictly less than $\eta_\lambda$,
  then the natural map
  \begin{equation}
    \HH^k(R,\mathcal{F})^G \to \HH^k(R^{\chi\semistable},\mathcal{F})^G
  \end{equation}
  is an isomorphism for all $k \geq 0$.
\end{theorem}

In particular, if~$G$ acts freely on the semistable locus,
then we obtain isomorphisms
\begin{equation}
  \HH^k(R,\mathcal{F})^G \to \HH^k(R{/\!/\!_\chi}G,\mathcal{F})
\end{equation}
for all~$k\geq 0$,
where on the right-hand side~$\mathcal{F}$ denotes the descent to the GIT quotient.

\begin{remark}
  \label{remark:weights}
  Note that as~$Z_\lambda$ is supposed to be connected,
  and $\Gm$ acts trivially via~$\lambda$ on~$Z_\lambda$,
 ~\cite[Proposition 2.10]{MR3329192} gives a split exact sequence
  \begin{equation}
    0 \to \mathbb{Z} \to \operatorname{Pic}^{\Gm}(Z_\lambda) \to \operatorname{Pic}(Z_\lambda) \to 0.
  \end{equation}
  A section is given by equipping a line bundle with the trivial linearization.
  The corresponding retraction is the map $\weight_\lambda\colon \operatorname{Pic}^{\Gm}(Z_\lambda) \to \mathbb{Z}$.
  This means that a $\Gm$-linearized line bundle $L$ on $Z_\lambda$
  is the same as a line bundle on $Z_\lambda$ together with
  a linear action of $\Gm$ on every fiber by the same weight $\weight_\lambda(L)$.
\end{remark}

\subsection{The Harder--Narasimhan and Hesselink stratifications for quiver moduli}
\label{section:HN-vs-Hesselink}
As in \cref{section:moduli}
we let~$Q$ be a quiver, $\tuple{d}$ a dimension vector
and~$\theta \in\Hom(\mathbb{Z}^{Q_0},\mathbb{Z})$ a stability parameter
such that $\theta(\tuple{d})=0$.
The character~$\chi_\theta$ of~$\group$ defined by~$\chi_\theta(g) = \prod_{i \in Q_0} \det(g_i)^{-\theta_i}$
descends to a character of~$\PG_{\tuple{d}} = \group/\Delta$,
therefore it makes no difference to consider the action of~$\group$ instead of the action of~$\PG_{\tuple{d}}$.

We consider the action of $\PG_{\tuple{d}}$ on the representation space $R = \repspace{Q,\tuple{d}}$
and describe the Hesselink stratification in this case.

For quiver representations, the fifth-named author established in~\cite[Proposition 3.4]{MR1974891}
the existence of an analogous stratification of $R$, the \emph{Harder--Narasimhan stratification}.
Under the standing assumption that~$\field$ is algebraically closed these two will coincide,
which is the content of \cref{theorem:HN=hesselink}.
In the case of complex numbers,
Hoskins gives a proof in~\cite[Theorem 3.8]{MR3871820}.
Over an arbitrary algebraically closed field,
Zamora gives a proof in~\cite[Theorem~6.3]{MR3199484}.
We recall below how the identification goes,
and explain how all the auxiliary data are related,
because we will need this for later computations.

We make the representation-theoretic setup of \cref{subsection:hesselink-and-teleman}
explicit in our special setting.
Let~$T_{\tuple{d}} \subseteq \group$ be the maximal torus of diagonal matrices.
The lattice $\cocharacters(T_{\tuple{d}})$ of one-parameter subgroups of~$T_{\tuple{d}}$
is freely generated by~$\{\lambda_{i,r} \mid i \in Q_0,\ r=1,\ldots,d_i\}$
where~$\lambda_{i,r}(z) \in T_{\tuple{d}}$
is the tuple of matrices where in the $i$th matrix there is a diagonal entry $z$
in the $r$th position.
Let~$(-,-)$ be the inner product on~$\cocharacters(T_{\tuple{d}})_\mathbb{R}$
for which~$\{\lambda_{i,r}\}$ is an orthonormal basis.
This inner product is invariant under the Weyl group~$W = \prod_{i \in Q_0}\mathrm{S}_{d_i}$.
Let $\|-\|$ be the corresponding norm.
Note that,
for each $i \in Q_0$,
the restriction of this norm on $\cocharacters(T_{d_i})$ is the standard Euclidean norm.

We also have to consider the maximal torus~$T = T_{\tuple{d}}/\Delta$ of~$\PG_{\tuple{d}}$.
Its lattice of one-parameter subgroups is~$\cocharacters(T) = \cocharacters(T_{\tuple{d}})/\mathbb{Z}\delta$,
where $\delta = \sum_{i\in Q_0}\sum_{r=1}^{d_i}\lambda_{i,r}$,
or more concretely $\delta(z) = z\cdot \operatorname{id}$.
We may identify~$\cocharacters(T)_\mathbb{R}$ with
the orthogonal complement~$(\mathbb{R}\delta)^\bot$ inside~$\cocharacters(T_{\tuple{d}})_\mathbb{R}$.
It consists of all real linear combinations~$\sum_{i\in Q_0}\sum_{r=1}^{d_i} a_{i,r}\lambda_{i,r}$
such that $\sum_{i\in Q_0}\sum_{r=1}^{d_i} a_{i,r} = 0$.
We restrict the norm $\|-\|$ to this subspace.

Now let us briefly introduce the terminology for the Harder--Narasimhan stratification.

For a dimension vector $\mathbf{e} \neq 0$ we define $|\mathbf{e}| = \sum_{i \in Q_0} e_i$
and the \emph{slope}~$\mu(\mathbf{e}) = \mu_\theta(\mathbf{e}) \colonequals \frac{\theta\cdot\mathbf{e}}{|\mathbf{e}|}$.
If~$M$ is a representation of~$Q$,
we write~$\mu(M)\colonequals\mu(\dimvect(M))$ and call it the slope of $M$.
Using the slope, we can generalize the notion of $\theta$-(semi\nobreakdash-)stability
to representations whose dimension vectors do not lie in the kernel of $\theta$.

\begin{definition}
  Let~$\mu=\mu_\theta$ be as above.
  A representation $M$ of $Q$ is called \emph{$\mu$-semistable} (respectively \emph{$\mu$-stable})
  if any nonzero proper subrepresentation $M'$ satisfies the inequality
  \begin{equation}
    \mu(M') \leq \mu(M) ~(\text{respectively } \mu(M') < \mu(M)).
  \end{equation}
\end{definition}

This generalized notion of semistability is used to define the Harder--Narasimhan stratification.
Let $M$ be a representation of $Q$.
A \emph{Harder--Narasimhan filtration} of $M$ is a sequence
\begin{equation}
  0 = N^0 \subset N^1 \subset N^2 \subset \ldots \subset N^\ell = M
\end{equation}
of subrepresentations such that each subquotient $N^m/N^{m-1}$ is $\mu$-semistable
and such that the chain of inequalities
\begin{equation}
  \mu(N^1/N^0) > \mu(N^2/N^1) > \ldots > \mu(N^\ell/N^{\ell-1})
\end{equation}
holds.
We call $\tuple{d}^* = (\dimvect(N^1/N^0),\ldots,\dimvect(N^{\ell}/N^{\ell-1}))$ the \emph{type} of the filtration,
and~$\ell$ its length.
Then Rudakov establishes in \cite[Theorem~3]{MR1480783} the existence
and uniqueness
of the Harder--Narasimhan filtration,
see also~\cite[Theorem 2.5]{MR1906875} which proves this directly
in the case of quiver representations.

We may stratify~$\repspace{Q,\tuple{d}}$ by Harder--Narasimhan type as follows.
Let~$\tuple{d}^* = (\tuple{d}^1,\ldots,\tuple{d}^\ell)$ be a Harder--Narasimhan type,
i.e.,~a sequence of dimension vectors of strictly decreasing slope
such that for each dimension vector~$\tuple{d}^m$
there exists a~$\mu$-semistable representation of dimension vector~$\tuple{d}^m$.

For a representation $M \in \repspace{Q,\tuple{d}}$
and~$a\in Q_1$,
we can decompose the matrix~$M_a\colon M_{\source(a)}\to M_{\target(a)}$ into blocks
\begin{equation}
  \label{equation:shape-matrix-HN-blocks}
  M_a =
  \left(
    \begin{matrix}
      M_a^{1,1}    & \hdots & M_a^{1,\ell} \\
      \vdots       & \ddots & \vdots \\
      M_a^{\ell,1} & \hdots & M_a^{\ell,\ell}
    \end{matrix}
  \right),
\end{equation}
where each block $M_a^{n,m}$ is of size $d^n_{\target(a)} \times d^m_{\source(a)}$,
for~$m,n=1,\ldots,\ell$.
Similarly, each component~$g_i$ of~$g \in \group$
can be decomposed into blocks~$g_i^{n,m}$ of size~$d^n_i \times d^m_i$,
for~$m,n=1,\ldots,\ell$.

Let $L_{\tuple{d}^*}$ and $P_{\tuple{d}^*}$ be subgroups of ${\group}$ defined by
\begin{align}
  L_{\tuple{d}^*} &\colonequals \{ g \in \group \mid g_i^{n,m} = 0 \text{ for all $i \in Q_0$ and all $n \neq m$}\}, \\
  P_{\tuple{d}^*} &\colonequals \{ g \in \group \mid g_i^{n,m} = 0 \text{ for all $i \in Q_0$ and all $n > m$}\}.
\end{align}
The group $L_{\tuple{d}^*}$ is a Levi factor of the parabolic subgroup $P_{\tuple{d}^*}$ of $\group$.
Let~$R_{\tuple{d}^*}$ and~$R_{\tuple{d}^*}^+$ be closed subvarieties of~$R$ defined by
\begin{align}
  R_{\tuple{d}^*}   &\colonequals \{ M \in \repspace{Q,\tuple{d}} \mid M_a^{n,m} = 0 \text{ for all $a \in Q_1$ and all $n \neq m$}\}, \\
  \label{equation:upper-triangular}
  R_{\tuple{d}^*}^+ &\colonequals \{ M \in \repspace{Q,\tuple{d}} \mid M_a^{n,m} = 0 \text{ for all $a \in Q_1$ and all $n > m$}\}.
\end{align}
The group~$L_{\tuple{d}^*} \cong \group[\tuple{d}^1] \times \ldots \times \group[\tuple{d}^\ell]$
acts on~$R_{\tuple{d}^*} \cong \repspace{Q,\tuple{d}^1} \times \ldots \times \repspace{Q,\tuple{d}^\ell}$,
and~$P_{\tuple{d}^*}$ acts on~$R_{\tuple{d}^*}^+$.
The projection~$p_{\tuple{d}^*}\colon R_{\tuple{d}^*}^+ \to R_{\tuple{d}^*}$
which forgets the off-diagonal blocks is equivariant with respect to the projection~$P_{\tuple{d}^*} \to L_{\tuple{d}^*}$.

\begin{definition}
  \label{definition:harder-narasimhan-strata}
  Let $\tuple{d}^* = (\tuple{d}^1,\ldots,\tuple{d}^\ell)$ be a Harder--Narasimhan type.
  \begin{enumerate}
    \item The locus $R_{\tuple{d}^*}^{\HN} = \{ M \in \repspace{Q,\tuple{d}} \mid \text{$M$ has a Harder--Narasimhan filtration of type $\tuple{d}^*$}\}$
      is called the associated \emph{Harder--Narasimhan stratum}.
    \item We define $Z_{\tuple{d}^*} \colonequals \repspace[\mu\semistable]{Q,\tuple{d}^1} \times \ldots \times \repspace[\mu\semistable]{Q,\tuple{d}^\ell}$.
    \item We define $\blade[\tuple{d}^*] \colonequals p_{\tuple{d}^*}^{-1}(Z_{\tuple{d}^*})$.
  \end{enumerate}
\end{definition}

We state the fifth-named author's result~\cite[Proposition 3.4]{MR1974891} on the Harder--Narasimhan stratification.

\begin{theorem}[Reineke]
  \label{theorem:HN-stratification}
  The affine space $R$ admits a decomposition
  \begin{equation}
    R = \bigsqcup_{\tuple{d}^*} R_{\tuple{d}^*}^{\HN},
  \end{equation}
  as a finite disjoint union into finitely many locally closed irreducible $\group$-invariant subsets. Moreover,
  \begin{enumerate}
    \item The set $Z_{\tuple{d}^*}$ is an $L_{\tuple{d}^*}$-invariant open subset of $R_{\tuple{d}^*}$;
    \item The set $\blade[\tuple{d}^*]$ is a $P_{\tuple{d}^*}$-invariant open subset of $R_{\tuple{d}^*}^+$;
    \item Each stratum $R_{\tuple{d}^*}^{\HN}$ satisfies $R_{\tuple{d}^*}^{\HN} = {\group}\cdot \blade[\tuple{d}^*]$; and
    \item \label{item:action-map} The action map~$\sigma\colon\group \times \blade[\tuple{d}^*] \to R_{\tuple{d}^*}^{\HN}$
      induces an isomorphism with the associated fiber bundle,
      i.e.,~$R_{\tuple{d}^*}^{\HN} \cong \group \times^{P_{\tuple{d}^*}} \blade[\tuple{d}^*]$.
  \end{enumerate}
\end{theorem}

We are going to describe the identification of the two stratifications.
Let $\tuple{d}^* = (\tuple{d}^1,\ldots,\tuple{d}^\ell)$ be a Harder--Narasimhan type
and assume that $\tuple{d}^* \neq (\tuple{d})$.
Let $C$ be the minimal positive integer such that
\begin{equation}
  \label{equation:ks}
  k_m\colonequals C\mu(\tuple{d}^m) \in \mathbb{Z}
\end{equation}
for all $m=1,\ldots,\ell$.
In \cref{table:harder-narasimhan-strata} we have listed the Harder--Nararasimhan types
for an interesting Kronecker moduli space,
together with the values of~$\mu(\tuple{d}^m)$, $C$, and~$k_m$
for all types, and all~$m=1,\ldots,\ell$.

We define a one-parameter subgroup $\lambda = \lambda_{\tuple{d}^*} = (\lambda_i)_{i \in Q_0} \in \cocharacters(T_{\tuple{d}})$ by
\begin{equation}
  \label{equation:HN-type-1-PS-correspondence}
  \lambda_i(z) = \operatorname{diag}\big( \underbrace{z^{k_1},\ldots,z^{k_1}}_{d_i^1 \text{ times}}; \underbrace{z^{k_2},\ldots,z^{k_2}}_{d_i^2 \text{ times}};\ldots; \underbrace{z^{k_\ell},\ldots,z^{k_\ell}}_{d_i^\ell \text{ times}} \big).
\end{equation}
Note that $\lambda$ is primitive by the minimality of $C$,
and that
\begin{equation}
  (\delta,\lambda)
  =\sum_{i \in Q_0} \sum_{m=1}^\ell k_md_i^m
  =\sum_{m=1}^\ell k_m|\tuple{d}^m|
  =C\sum_{m=1}^\ell \theta(\tuple{d}^m)
  =C\theta(\tuple{d})
  = 0.
\end{equation}
This shows that $\lambda$ lies in $(\mathbb{R}\delta)^\bot$
which we have identified with $\cocharacters(T)_\mathbb{R}$.
We thus can, and will, interpret~$\lambda$ also as a 1-PS of~$\PG_{\tuple{d}}$.

We can now compare the two stratifications obtained in \cref{theorem:hesselink,theorem:HN-stratification}.
This result is proven independently in~\cite[Theorem 3.8]{MR3871820} and~\cite[Theorem~6.3]{MR3199484}.
For a representation~$M$ which has a Harder--Narasimhan filtration of type $\tuple{d}^*$,
we get $[\lambda] \cap \Lambda(M) \neq \emptyset$.
This shows that every primitive one-parameter subgroup~$\lambda \in \cocharacters(T)$
which occurs in the Hesselink stratification for $\repspace{Q,\tuple{d}}$
is of the form~$\lambda_{\tuple{d}^*}$ after conjugation with a suitable Weyl group element
for a unique Harder--Narasimhan type $\tuple{d}^*$.

\begin{theorem}[Hesselink--Harder--Narasimhan correspondence]
  \label{theorem:HN=hesselink}
  For the Hesselink and Harder--Nara\-sim\-han stratifications
  from \cref{theorem:hesselink,theorem:HN-stratification}
  \begin{equation}
    R \setminus R^{\theta\semistable} = \bigsqcup_{\tuple{d}^* \neq (\tuple{d})} R_{\tuple{d}^*}^{\HN} = \bigsqcup_{[\lambda]} S_{[\lambda]}
  \end{equation}
  the following hold:
  \begin{enumerate}
    \item For every Harder--Narasimhan type $\tuple{d}^* \neq (\tuple{d})$,
      the one-parameter subgroup~$\lambda_{\tuple{d}^*} \in \cocharacters(T)$
      satisfies~$\langle \chi_\theta,\lambda_{\tuple{d}^*} \rangle < 0$
      and~$R_{\tuple{d}^*}^{\HN} = S_{[\lambda_{\tuple{d}^*}]}$.
    \item For ever ${\group}$-conjugacy class $[\lambda]$ in the Hesselink stratification for which~$S_{[\lambda]} \neq \emptyset$,
      there exists a unique Harder--Narasimhan type $\tuple{d}^* \neq (\tuple{d})$
      such that $[\lambda] = [\lambda_{\tuple{d}^*}]$.
  \end{enumerate}
\end{theorem}
The same correspondence also gives rise to several other identifications,
that we state and introduce notation for in the following remark.
\begin{remark}
  \label{remark:technicalities-and-abbreviations-HN-Hesselink-equivalence}
  Let $\tuple{d}^* \neq (\tuple{d})$ be a Harder--Narasimhan type
  and let $\lambda = \lambda_{\tuple{d}^*}$ be the one-parameter subgroup of $T\subset\PG_{\tuple{d}}$
  which corresponds to it under \cref{theorem:HN=hesselink}.
  Then we set
  \begin{align}
    L      &\colonequals L_{\tuple{d}^*}/\Delta = L_\lambda \\
    P      &\colonequals P_{\tuple{d}^*}/\Delta = P_\lambda \\
    R^0    &\colonequals R_{\tuple{d}^*} = R_\lambda \\
    R^+    &\colonequals R_{\tuple{d}^*}^+ = R_\lambda^+ \\
    Z      &\colonequals Z_{\tuple{d}^*} = Z_\lambda \\
    \blade &\colonequals \blade[\tuple{d}^*] = \blade[\lambda] \\
    S      &\colonequals R_{\tuple{d}^*}^{\HN} = S_{[\lambda]}.
  \end{align}
  Then we have the following situation:
  \begin{equation}
    \begin{tikzcd}
      L & P \\[-1.5em]
      \curvearrowright & \curvearrowright \\[-1.5em]
      R^0 & R^+ \arrow{l}{p} \\[-1.5em]
      \text{\rotatebox[origin=c]{90}{$\subseteq$}} & \text{\rotatebox[origin=c]{90}{$\subseteq$}} \\[-1.5em]
      Z & \Sigma \arrow{l}{} \arrow{r}{} \arrow{d}{} & S \arrow{d}{} \\
      & \{eP\} \arrow{r}{} & {\group}/P.
    \end{tikzcd}
  \end{equation}
\end{remark}

\subsection{Width of the windows for quiver moduli}
\label{subsection:width}
Let~$\tuple{d}^*=(\tuple{d}^1,\ldots,\tuple{d}^\ell)$ be a Harder--Narasimhan type,
and let $\lambda \colonequals \lambda_{\tuple{d}^*} \in \characters_*({\group})$ be
the corresponding one-parameter subgroup given by~\eqref{equation:HN-type-1-PS-correspondence}.
Let $S$ be the Hesselink stratum associated to $\lambda\colonequals\lambda_{\tuple{d}^*}$,
or equivalently let it be the Harder--Narasimhan stratum associated to $\tuple{d}^*$,
let $\blade\colonequals\blade[\lambda]$ be the blade of $\lambda$ (and of $\tuple{d}^*$),
and let~$\limitset\colonequals\limitset[\lambda]$ be the limit set of $\lambda$ (and of $\tuple{d}^*$).

We wish to compute the weight of $\det(\normal_{S/R}^\vee|_Z)$
for the action of $\lambda$,
which we denote by $\eta_\lambda$,
which is one of the ingredients in the statement of Teleman quantization as in \cref{theorem:teleman-quantization}.
By using the equivariant adjunction formula,
which we state for a general case below,
we will be able to split the computation in two parts and conduct each of them separately.
\begin{lemma}
  \label{lemma:determinant-normal-bundle-sequence}
  Let $R$ be a smooth variety,
  and let $S$ be a smooth, locally closed subvariety of $R$.
  Let a linearly reductive algebraic group $G$ act on $R$,
  and assume that $S$ is $G$-stable.
  In this case, the standard adjunction isomorphism
  \begin{equation}
    \det(\normal_{S/R})^\vee \cong \omega_R|_S \otimes \omega_S^\vee
  \end{equation}
  is $G$-equivariant, i.e., it holds in $\Pic^G(S)$.
\end{lemma}

\begin{proof}
  By assumption $S$ is locally closed,
  hence closed in an open $U$ of $R$.
  Since both $U$ and $S$ are smooth,
  the standard adjunction formula applies. %
  This isomorphism is $G$-equivariant because it is induced from the short exact sequence
  \begin{equation}
    \label{equation:equivariant-adjunction}
    0 \to \mathcal{I}_{S}/\mathcal{I}^2_{S} \to \Omega^{1}_{U/\field}|_{S} \to \Omega^{1}_{S/\field} \to 0,
  \end{equation}
  which is $G$-equivariant as $S$ is $G$-stable.
\end{proof}

The space $R = \repspace{Q,\tuple{d}}$ is defined to be the product
of affine spaces~$\Mat_{d_{\target(a)} \times d_{\source(a)}}$,
indexed by the arrows~$a\in Q_1$, see \eqref{equation:representation-variety}. The coordinate ring of $R$ is thus
\begin{equation}
  \mathcal{O}(R) = \field[x_{p,q}^{(a)}]_{a \in Q_1, ~p=1,\ldots,d_{\target(a)}, ~q=1,\ldots,d_{\source(a)}},
\end{equation}
where $x_{p,q}^{(a)}$ is the regular function which selects from a tuple $M = (M_a)_{a \in Q_1}$ the $(p,q)$-th entry of the matrix $M_a$.
The sheaf of differentials $\Omega^{1}_R$ is therefore the sheaf associated to the free $\mathcal{O}(R)$-module with basis $\{\diff x_{p,q}^{(a)} \mid a \in Q_1, ~p=1,\ldots,d_{\target(a)}, ~q=1,\ldots,d_{\source(a)}\}$.

The action of $\group$ on $R\colonequals\repspace{Q,\tuple{d}}$ induces a left action on the coordinate ring,
by precomposition with the inverse, so for $f \in \mathcal{O}(R)$ and $g = (g_i)_{i \in Q_0} \in \group$,
the regular function $g\cdot f$ is defined by
\begin{equation}
  (g\cdot f)(M) = f(g^{-1}\cdot M) = f((g_{\target(a)}^{-1}M_ag_{\source(a)})_{a \in Q_1})
\end{equation}
for all $M = (M_a)_{a \in Q_1} \in R$.
The induced left action of $\group$ on the sheaf of differentials $\Omega^1_R$
is given by~$g \cdot \mathrm{d}f \colonequals \mathrm{d}(g \cdot f)$, for all $f \in \mathcal{O}(R)$.
This yields left actions on all exterior products and all $\group$-stable subsheaves of $\Omega^1_R$,
as well as on quotients of $\Omega^1_R$ by such subsheaves.

First we deal with the first tensor factor in \eqref{equation:equivariant-adjunction}
restricted to~$Z$,
i.e., we consider the limit set~$Z$ inside stratum~$S$
which itself lives inside the representation variety~$R$.
\begin{lemma}
  \label{lemma:omega-RZ-weight}
  The $\lambda$-weight of the canonical bundle~$\omega_{R}|_Z$ on~$R$ restricted is
  \begin{equation}
    \label{equation:omega-RZ-weight}
    \weight_{\lambda}(\omega_R|_Z) = \sum_{1\leq m<n\leq\ell}(k_n-k_m)\left( \langle \tuple{d}^m,\tuple{d}^n\rangle - \langle \tuple{d}^n,\tuple{d}^m\rangle \right).
  \end{equation}
\end{lemma}

\begin{proof}
  Recall that $\lambda = (\lambda_i)_{i \in Q_0}$ consists of diagonal one-parameter subgroups $\lambda_i$,
  see~\eqref{equation:HN-type-1-PS-correspondence}.
  The action of~$\lambda$ on~$x_{p,q}^{(a)}$ is given by
  \begin{equation}
    (\lambda(z) \cdot x_{p,q}^{(a)})(M) = x_{p,q}^{(a)}((\lambda_{\target(b)}(z)^{-1}M_b\lambda_{\source(b)}(z))_{b \in Q_1})
    = \lambda_{\target(a)}(z)_{p,p}^{-1}\lambda_{\source(a)}(z)_{q,q}\cdot x_{p,q}^{(a)}(M)
  \end{equation}
  for all $M = (M_b)_{b \in Q_1}$ and all $z \in \Gm$.
  This shows that
  \begin{equation}
    \lambda(z)\cdot \diff x_{p,q}^{(a)} =  \lambda_{\target(a)}(z)_{p,p}^{-1}\lambda_{\source(a)}(z)_{q,q} \diff x_{p,q}^{(a)}.
  \end{equation}
  Taking the exterior product over all the generators $\diff x_{p,q}^{(a)}$,
  we obtain
  \begin{equation}
    \bigwedge_{a \in Q_1}\bigwedge_{\substack{1 \leq p \leq d_{\target(a)} \\ 1 \leq q \leq d_{\source(a)}}} \left( \lambda(z)\cdot \diff x_{p,q}^{(a)} \right)
    = \left( \prod_{a \in Q_1} \prod_{\substack{1 \leq p \leq d_{\target(a)} \\ 1 \leq q \leq d_{\source(a)}}} \lambda_{\source(a)}(z)_{q,q} \lambda^{-1}_{\target(a)}(z)_{p,p} \right) \bigwedge_{a \in Q_1} \bigwedge_{\substack{1 \leq p \leq d_{\target(a)} \\ 1 \leq q \leq d_{\source(a)}}} \diff x_{p,q}^{(a)}.
  \end{equation}
  This determines the character of the $\Gm$-linearization via $\lambda$ of $\omega_R$. Its weight is necessarily $\weight_\lambda(\omega_R|_Z)$. So, we obtain
  \begin{equation}
    \begin{aligned}
      \weight_\lambda(\omega_R|_Z)
      &= \sum_{a \in Q_1} \sum_{1 \leq m,n \leq \ell} (k_m - k_n) d_{\source(a)}^m d_{\target(a)}^n \\
      &= \sum_{a \in Q_1} \sum_{1 \leq m < n \leq \ell}(k_m - k_n)(d_{\source(a)}^m d_{\target(a)}^n - d_{\source(a)}^n d_{\target(a)}^m) \\
      &= \sum_{1 \leq m < n \leq \ell}(k_m - k_n)(\langle \tuple{d}^n,\tuple{d}^m \rangle - \langle \tuple{d}^m,\tuple{d}^n \rangle) \\
      &= \sum_{1 \leq m < n \leq \ell}(k_n - k_m)(\langle \tuple{d}^m,\tuple{d}^n \rangle - \langle \tuple{d}^n,\tuple{d}^m \rangle).
    \end{aligned}
  \end{equation}
\end{proof}

Next we deal with the second tensor factor in \eqref{equation:equivariant-adjunction}
restricted to~$Z$,
i.e., we consider the limit set~$Z$ inside the stratum~$S$,
without reference to~$R$.
\begin{lemma}
  \label{lemma:omega-SZ-weight}
  The~$\lambda$-weight of~$\omega_S|_Z$ is
  \begin{equation}
    \label{equation:omega-SZ-weight}
    \weight_{\lambda}(\omega_S|_Z) = \sum_{1\leq m<n\leq\ell}(k_m-k_n)\langle \tuple{d}^n,\tuple{d}^m\rangle.
  \end{equation}
\end{lemma}
To give a proof, we first explain how to split the computation in several steps,
eventually leading to the identity in~\eqref{equation:split-computation},
and then perform each step separately.

The morphism~$\sigma\colon \group \times \blade \to S$
in \cref{theorem:HN-stratification}(\ref{item:action-map})
is a principal fiber bundle (for the \'etale topology,
by the standing assumption on~$\field$ being of characteristic zero)
with fiber $P_{\mathbf{d}^*}$.
Therefore, $\sigma$ is smooth and the relative tangent bundle sequence
\begin{equation}
  \label{equation:relative-tangent-bundle-sequence}
  0
  \to\tangent_{\group \times \blade/S}
  \to\tangent_{\group \times \blade}
  \to\sigma^*\tangent_{S}
  \to0.
\end{equation}
is exact.
The fiber of $\sigma$ in a point $\sigma(g,M) = g\cdot M$ is the $P_{\mathbf{d}^*}$-orbit of $(g,M)$.
The latter is isomorphic to~$P_{\mathbf{d}^*}$ via the action of $P_{\mathbf{d}^*}$ on the point $(g,M)$, which is given by
\begin{equation}
  p \cdot (g,M) = (gp^{-1}, p\cdot M).
\end{equation}
The sequence \eqref{equation:relative-tangent-bundle-sequence} thus translates to
\begin{equation}
  \label{equation:relative-tangent-bundle-sequence-rewritten}
  0
  \to\mathfrak{p}_{\mathbf{d}^*} \otimes \mathcal{O}_{\group \times \blade}
  \to\tangent_{\group \times \blade}
  \to\sigma^*\tangent_{S}
  \to0,
\end{equation}
where, over a point $(g,M) \in \group \times \blade$,
the map $\mathfrak{p}_{\mathbf{d}^*} \to T_{\group \times \blade,(g,M)} = g\cdot \mathfrak{g} \oplus T_{\blade,M}$
is the derivative of the  map $P_{\mathbf{d}^*} \to \group \times \blade$
defined by $p \mapsto p\cdot (g,M) = (gp^{-1}, p\cdot M)$.

The sequence \eqref{equation:relative-tangent-bundle-sequence}
is $P_{\mathbf{d}^*}$-equivariant, which implies that~%
\eqref{equation:relative-tangent-bundle-sequence-rewritten} is also $P_{\mathbf{d}^*}$-equivariant;
note that here we consider the adjoint action of~$P_{\mathbf{d}^*}$ on~$\mathfrak{p}_{\mathbf{d}^*}$, which is induced by the conjugation action of $P_{\tuple{d}^*}$ on $\group$, the left multiplication on $\Sigma$ and the action induced by $\sigma$ on~$\tangent_{S}$.

Now we restrict the action of $P_{\mathbf{d}^*}$ to an action of $\Gm$
via the one-parameter subgroup $\lambda\colon \Gm \to P_{\mathbf{d}^*} \subseteq \group$.
As $Z$ is the locus of fixed points of the $\lambda$-action,
$\lambda$ acts on every fiber of $T_S|_Z$,
and as $Z$ is connected, the action is the same on every fiber.
Let $M \in Z$ and consider the sequence \eqref{equation:relative-tangent-bundle-sequence-rewritten}
in the fiber of the point~$(e,M) \in \group \times \blade$.
It is
\begin{equation}
  \label{equation:relative-tangent-bundle-sequence-fiber}
  0
  \to\mathfrak{p}_{\mathbf{d}^*}
  \to\mathfrak{g}_{\mathbf{d}} \oplus \tangent_{\blade,M}
  \to\tangent_{S,M}
  \to0.
\end{equation}
The point $(e,M)$ is a fixed point for the $\lambda$-action on $\group \times \blade$,
which implies that $\eqref{equation:relative-tangent-bundle-sequence-fiber}$ is a short exact sequence of representations of $\Gm$.

As the blade $\blade$ is open in the affine space $R_{\mathbf{d}^*}^+$, see \eqref{equation:upper-triangular},
the tangent space to $\blade$ at every point identifies
with $R_{\mathbf{d}^*}^+$ (considered as a vector space).
The exact sequence \eqref{equation:relative-tangent-bundle-sequence-fiber} thus becomes
\begin{equation}
  \label{equation:relative-tangent-bundle-sequence-fiber-rewritten}
  0
  \to\mathfrak{p}_{\mathbf{d}^*}
  \to\mathfrak{g}_{\mathbf{d}} \oplus R_{\mathbf{d}^*}^+
  \to\tangent_{S,M}
  \to0.
\end{equation}
Although it is not relevant for the weight computation,
let us give the map $\mathfrak{p}_{\mathbf{d}^*} \to \mathfrak{g}_{\mathbf{d}} \oplus R_{\mathbf{d}^*}^+$ explicitly.
It sends $x \in \mathfrak{p}_{\mathbf{d}^*}$ to $(-x, [x,M])$, where
\begin{equation}
  [x,M] = (x_{\target(a)}M_a - M_ax_{\source(a)})_{a \in Q_1}.
\end{equation}

We use the sequence \eqref{equation:relative-tangent-bundle-sequence-fiber-rewritten}
to compute the weight of the restriction to $Z$ of the anticanonical bundle of $S$.
It is
\begin{equation}
  \label{equation:split-computation}
  \weight_{\lambda}(\omega_S^\vee|_Z) = \weight_{\lambda}(\det(\mathfrak{g}_{\mathbf{d}})) + \weight_{\lambda}(\det(R_{\mathbf{d}^*}^+)) - \weight_{\lambda}(\det(\mathfrak{p}_{\mathbf{d}^*})).
\end{equation}

Now to the individual weights in the right-hand side of~\eqref{equation:split-computation}.
The one-parameter subgroup~$\lambda$ acts on $\mathfrak{g}_{\mathbf{d}}$ by conjugation
and $\mathfrak{p}_{\mathbf{d}^*}$ is a submodule.
On $R_{\mathbf{d}^*}^+$,
it acts by restriction of the $\group$-action via~$\lambda\colon \Gm \to \group$.
As~$\lambda$ consists of diagonal matrices,
the matrix entries of elements of $\mathfrak{g}_{\mathbf{d}}$
and of $R_{\mathbf{d}^*}^+$ are weight spaces.
Using the decomposition of the summands of $R_{\mathbf{d}^*}^+$
into blocks as in \eqref{equation:shape-matrix-HN-blocks}
and similarly for the summands of $\mathfrak{g}_{\mathbf{d}}$ and $\mathfrak{p}_{\mathbf{d}^*}$,
we obtain the following.
Here, $\field(r)$ is the one-dimensional $\Gm$-representation
whose weight is $r$.
\begin{lemma}
  \label{lemma:weights-fibers}
  Regarded as $\Gm$-representations via $\lambda$, we have the following isomorphisms:
  \begin{align}
    \mathfrak{g}_{\mathbf{d}} &\cong \bigoplus_{i \in Q_0} \bigoplus_{1 \leq m,n \leq \ell} \field(k_m-k_n)^{d_i^md_i^n} \\
    \mathfrak{p}_{\mathbf{d}^*} &\cong \bigoplus_{i \in Q_0} \bigoplus_{1 \leq m<n \leq \ell} \field(k_m-k_n)^{d_i^md_i^n} \\
    R_{\mathbf{d}^*}^+ &\cong \bigoplus_{a \in Q_1} \bigoplus_{1 \leq m < n \leq \ell} \field(k_m-k_n)^{d_{\target(a)}^md_{\source(a)}^n}
  \end{align}
\end{lemma}
The above lemma implies at once the following.
\begin{lemma}
  \label{lemma:weights-determinants}
  The $\lambda$-weights of the determinants of $\mathfrak{g}_{\mathbf{d}}$, $\mathfrak{p}_{\mathbf{d}^*}$, and $R_{\mathbf{d}^*}^+$ are
  \begin{align}
    \weight_\lambda(\det(\mathfrak{g}_{\mathbf{d}})) &= 0 \\
    \weight_\lambda(\det(\mathfrak{p}_{\mathbf{d}^*})) &= \sum_{1 \leq m < n \leq \ell} (k_m - k_n) \left(\sum_{i \in Q_0} d_i^n  d_i^m\right) \\
    \weight_\lambda(\det(R_{\mathbf{d}^*}^+)) &= \sum_{1 \leq m < n \leq \ell} (k_m - k_n) \left(\sum_{a\in Q_1} d_{\target(a)}^m d_{\source(a)}^n\right).
  \end{align}
\end{lemma}

We can now give the proof of \cref{lemma:omega-SZ-weight}.
\begin{proof}[Proof of \cref{lemma:omega-SZ-weight}]
  Using~\eqref{equation:split-computation},
  we conclude that the $\lambda$-weight of $\omega_S^\vee|_Z$ is
  \begin{equation}
    \begin{aligned}
      \weight_{\lambda}(\omega_S^\vee|_Z)
      &= \sum_{1 \leq m < n \leq \ell} (k_m - k_n) \left(\sum_{a \in Q_1} d_{\source(a)}^n d_{\target(a)}^m - \sum_{i \in Q_0} d_i^n d_i^m\right) \\
      &= \sum_{1\leq m < n \leq \ell} (k_n - k_m) \langle \tuple{d}^n,\tuple{d}^m \rangle.
    \end{aligned}
  \end{equation}
\end{proof}
Thus from \cref{lemma:determinant-normal-bundle-sequence}
and the computations in \cref{lemma:omega-RZ-weight,lemma:omega-SZ-weight}
we obtain the first ingredient to apply Teleman quantization.
\begin{corollary}
  \label{corollary:determinant-weight}
  The $\lambda$-weight of~$\det\big(\normal_{S/R}^\vee|_Z\big)$ is
  \begin{equation}
    \label{equation:weight-eta-notation}
    \eta_\lambda\colonequals\weight_{\lambda}(\det\normal_{S/R}^\vee|_Z) = \sum_{1 \leq m < n \leq \ell}(k_n - k_m)\langle \tuple{d}^m,\tuple{d}^n \rangle.
  \end{equation}
\end{corollary}
Throughout we will denote this weight by $\eta_\lambda$.
In \cref{table:harder-narasimhan-strata} we have computed~$\eta_\lambda$
for all Harder--Narasimhan strata,
in our running example of an interesting Kronecker moduli space.

\subsection{Weights of the endomorphism bundle of the universal bundle}
\label{subsection:weights-endomorphism-bundle}
Recall from \cref{section:moduli} that if $\tuple{d}$ is $\theta$-coprime, $\modulispace[\theta\stable]{Q,\tuple{d}}$
comes equipped with a universal bundle~$\mathcal{U}=\mathcal{U}(\tuple{a})$.
As explained in \cref{remark:uniqueness-up-to-twist-universal-bundle},
this bundle is unique up to the choice of a normalization,
which is given by a tuple~$\tuple{a}\in\mathbb{Z}^{Q_0}$ such that~$\tuple{a}\cdot\tuple{d}=1$.

The bundles~$\mathcal{U}_i(\tuple{a})$ are the descent of
(the restriction to the stable locus of)
the~$\group$-equivariant bundles~$U_i(\tuple{a})=U_i\otimes L(\tuple{a})$
on~$R=\repspace{Q,\tuple{d}}$.
We will now consider these equivariant bundles on~$R$ itself,
thus in what follows,
we do not have to put any conditions to ensure their existence,
or choose any normalizations.

For a Harder--Narasimhan type $\tuple{d}^*$,
we will compute the weights of the action of $\lambda_{\tuple{d}^*}$
on the universal bundles $U_i(\tuple{a})$ on $R$.

Recall that the element $g \in {\group}$ acts on $u_i \in U_i(\tuple{a})$ as
\begin{equation}
  g \cdot u_i \colonequals \left(\prod_{j \in Q_0} \det(g_j)^{-a_j} \right) g_i u_i.
\end{equation}
We have defined the integers~$k_m$ for~$m=1,\ldots,\ell$ in~\eqref{equation:ks}
as the smallest integer multiples of~$\mu(\tuple{d}^m)$,
and used it to define the~1-parameter subgroup $\lambda_{\tuple{d}^*}$ in~\eqref{equation:HN-type-1-PS-correspondence}.
From the block decomposition in~\eqref{equation:HN-type-1-PS-correspondence}
we obtain that~$z\in\lambda_{\tuple{d}^*}$ acts by
\begin{equation}
  \begin{aligned}
    \label{equation:1-PS-acting-on-Ui(a)}
    z \cdot u_i
    &= \left(\prod_{j\in Q_0}\det(\lambda_j(z))^{-a_j}\right)\lambda_i(z)u_i \\
    &= \left(\prod_{j\in Q_0} z^{-a_j\sum_{n=1}^\ell d^n_j k_n}\right) \lambda_i(z) u_i.
  \end{aligned}
\end{equation}

\begin{lemma}
  The weights of the action of~$\lambda$
  on $U_i(\tuple{a})$ are
  \begin{equation}
    \label{equation:weights-Ui(a)}
    \left\{ k_m - \sum_{j \in Q_0} \sum_{n=1}^\ell a_j d^n_j k_n \mid 1\leq m \leq \ell \right\},
  \end{equation}
  where the weight indexed by~$m$ appears with multiplicity~$d_i^m$.
\end{lemma}

\begin{proof}
  It suffices to observe in~\eqref{equation:1-PS-acting-on-Ui(a)}
  that the weights of~$\lambda_i(z)$ acting on~$u_i$ are~$k_1,\ldots,k_\ell$ and the dimension of the weight space of weight $k_m$ is $d_i^m$.
\end{proof}

The choice of normalization disappears when we consider the summands of the endomorphism bundle,
as in the statement of the following proposition.

\begin{proposition}
  Let~$Q$ be a quiver,
  $\tuple{d}$ a dimension vector,
  and~$\theta$ a stability parameter.
  Let~$\tuple{d}^*$ be a Harder--Narasimhan type.

  The weights of the action of~$\lambda$
  on the $\group$-equivariant vector bundles~$U_i^\vee\otimes U_j=U_i^\vee\otimes L(-\tuple{a})\otimes U_j\otimes L(\tuple{a})$
  are given by
  \begin{equation}
    \{ k_m - k_n\mid 1 \leq m,n \leq \ell \},
  \end{equation}
  where~$k_m$ is defined as in~\eqref{equation:ks},
  and the weight~$k_m-k_n$ appears with multiplicity~$d_i^md_j^n$.
\end{proposition}

\begin{proof}
  It suffices to cancel the contribution of~$L(\tuple{a})$ in~\eqref{equation:weights-Ui(a)}
  and sum them together with opposite signs.
\end{proof}

\section{Main theorem and applications}
\label{section:results}

\subsection{Cohomology vanishing and Schofield's conjecture}
\label{subsection:cohomology-vanishing}
For any Harder--Narasimhan type $\tuple{d}^* = (\tuple{d}^1,\dots,\tuple{d}^\ell)$,
let $\lambda \colonequals \lambda_{\tuple{d}^*}$ be the corresponding 1-PS given in~\eqref{equation:HN-type-1-PS-correspondence};
let~$k_m \colonequals C\mu(\tuple{d}^m)=C\theta(\tuple{d}^m)/\lvert\tuple{d}^m\rvert$ be the smallest integer multiple of the slopes
as in~\eqref{equation:ks},
and let~$\eta_\lambda$ be the $\lambda$-weight of~$\det\big(\normal_{S/R}^\vee|_Z\big)$
as in \cref{corollary:determinant-weight}.

To prove the following proposition, we introduce the following stronger version of \cref{definition:ample-stability}.

\begin{definition}
  \label{definition:strong-ample-stability}
  The dimension vector $\tuple{d}$ is said to be \emph{strongly amply~$\theta$-stable}
  if for any dimension subvector $\tuple{e}$ for which $\mu(\tuple{e}) > \mu(\tuple{d} - \tuple{e})$,
  the inequality $\langle \tuple{e}, \tuple{d} - \tuple{e} \rangle \leq -2$ holds.
\end{definition}
The condition above has been shown~\cite[Proposition 5.1]{MR3683503} to be sufficient to imply ample stability of $\moduli$.
As we will see in \cref{example:strict-ample-stability}, the converse is not true.

\begin{proposition}
  \label{proposition:inequality}
  Let~$Q$ be a quiver,
  $\tuple{d}$ a dimension vector,
  and~$\theta$ a stability parameter,
  such that~$\tuple{d}$ is strongly amply~$\theta$-stable.
  Let~$\lambda=\lambda_{\tuple{d}^*}$ be the 1-PS subgroup
  corresponding to a Harder--Narasimhan type~$\tuple{d}^*$.
  For all $1 \leq m,n\leq\ell$, the inequality
  \begin{equation}
    \label{equation:teleman-inequality}
    k_m-k_n<\eta_\lambda
  \end{equation}
  holds.
\end{proposition}

\begin{proof}
  We can rewrite $\eta_\lambda$ as
  \begin{equation}
    \begin{aligned}
      \eta_\lambda
      &= \sum_{1\leq m < n\leq\ell} (k_n - k_m)\langle \tuple{d}^m,\tuple{d}^n \rangle \\
      &= \sum_{1\leq m < n\leq\ell} \sum_{r = m}^{n-1} (k_{r+1} - k_r) \langle \tuple{d}^m,\tuple{d}^n \rangle \\
      &= \sum_{r = 1}^\ell (k_{r+1} - k_r)\sum_{\substack{m \leq r \\ n > r}} \langle \tuple{d}^m, \tuple{d}^n \rangle \\
      &= \sum_{r = 1}^\ell (k_r - k_{r+1})\underbrace{\sum_{\substack{m \leq r \\ n > r}} \bigl( - \langle \tuple{d}^m, \tuple{d}^n \rangle\bigr)}_{=: N_r}.
    \end{aligned}
  \end{equation}
  To prove the inequality~\eqref{equation:teleman-inequality},
  it will be enough to show that~$N_r \geq 1$ for all~$r$,
  and that there exists at least one $r$ for which $N_r \geq 2$.
  To apply the condition from \cref{definition:strong-ample-stability},
  we notice that each term $N_r$ satisfies the equality
  \begin{equation}
    \begin{aligned}
      -N_r &= \left\langle \sum_{m \leq r} \tuple{d}^m, \sum_{n > r} \tuple{d}^n \right\rangle
      &= \langle \tuple{e}, \tuple{d} - \tuple{e} \rangle.
    \end{aligned}
  \end{equation}
  Since $\mu(\tuple{e}) > \mu(\tuple{d} - \tuple{e})$,
  we can conclude that $N_r \geq 2$.
  This completes the proof.
\end{proof}

Two remarks are now in order.

\begin{remark}
  \label{remark:hope}
  We hope that the strong ample stability in \cref{proposition:inequality}
  is not needed,
  and that ample stability in fact suffices,
  but we have not been able to prove this.
\end{remark}

\begin{remark}
  Without ample stability,~\eqref{equation:teleman-inequality} does not hold.
  Let~$\tuple{d}^*=(\tuple{d}^1,\ldots,\tuple{d}^\ell)$ be a Harder--Narasimhan type which has codimension~1.
  Its codimension can be computed as
  \begin{equation}
    \label{equation:codimension-one-sum}
    1=\sum_{1\leq m<n\leq\ell}-\langle\tuple{d}^m,\tuple{d}^n\rangle
  \end{equation}
  and because~$m<n$ in this sum, we have
  \begin{equation}
    \langle\tuple{d}^m,\tuple{d}^n\rangle
    =\operatorname{hom}(\tuple{d}^m,\tuple{d}^n)-\operatorname{ext}(\tuple{d}^m,\tuple{d}^n)
    =-\operatorname{ext}(\tuple{d}^m,\tuple{d}^n).
  \end{equation}
  Thus in the sum~\eqref{equation:codimension-one-sum}
  all but one term vanish,
  with one term having value exactly~1.
  But then
  \begin{equation}
    k_\ell-k_1\geq\sum_{1\leq m<n\leq\ell}(k_m-k_n)(-\langle\tuple{d}^m,\tuple{d}^n\rangle)=\eta_\lambda,
  \end{equation}
  where the last equality is \cref{corollary:determinant-weight},
  thus~\eqref{equation:teleman-inequality} does not hold.
\end{remark}

\Cref{proposition:inequality} allows us to prove the main technical result of the paper.

\begin{proof}[Proof of \cref{theorem:cohomology-vanishing}]
  In order to apply Teleman quantization
  as stated in \cref{theorem:teleman-quantization}
  we need to check for every~1-PS subgroup~$\lambda$
  arising in the Hesselink stratification that
  the limit set~$\limitset$ is connected,
  the weights of the equivariant bundle are bounded above by~$\eta_\lambda$ as defined in~\eqref{equation:eta-lambda}.
  The connectedness follows from the definition of~$\limitset[{\tuple{d}^*}]$
  in \cref{definition:harder-narasimhan-strata}
  and the correspondence in \cref{theorem:HN=hesselink}.
  The inequality is checked in \cref{proposition:inequality}.

  Thus we can apply Teleman quantization,
  and we obtain isomorphisms
  \begin{equation}
    \HH^k(\Rep(Q,\tuple{d}),U_i^\vee\otimes U_j)^{\group}
    \cong
    \HH^k(\Rep^{\theta\semistable}(Q,\tuple{d}),U_i^\vee\otimes U_j)^{\group}
    \cong
    \HH^k(\modulispace[\theta\stable]{Q,\tuple{d}},\mathcal{U}_i^\vee\otimes\mathcal{U}_j)
  \end{equation}
  for all~$k\geq 0$,
  as stability and semistability agree.
  But because the representation variety~$\Rep(Q,\tuple{d})$ is \emph{affine},
  the cohomology~$\HH^k(\Rep(Q,\tuple{d}),U_i^\vee\otimes U_j)$ vanishes for~$k\geq 1$,
  even before taking invariants,
  which proves~\eqref{equation:cohomology-vanishing}.
\end{proof}

This proves Schofield's conjecture as follows.
\begin{proof}[Proof of \cref{corollary:schofield}]
  This follows from the isomorphism
  \begin{equation}
    \Ext_{\modulispace[\theta\stable]{Q,\tuple{d}}}^{\geq 1}(\mathcal{U},\mathcal{U})
    \cong\bigoplus_{i,j\in Q_0}\HH^{\geq 1}(\modulispace[\theta\stable]{Q,\tuple{d}},\mathcal{U}_i^\vee\otimes\mathcal{U}_j),
  \end{equation}
  and the vanishing in \cref{theorem:cohomology-vanishing}.
\end{proof}

\paragraph{Windows and weights for a 6-dimensional Kronecker moduli space}
To illustrate the theory we will give the details of an interesting and relevant example.
We will consider the smallest Kronecker moduli space which is not a projective space or a Grassmannian,
and it was studied explicitly in \cite{2307.01711v2}.
\begin{example}
  \label{example:kronecker-moduli-space}
  Let~$Q$ be the 3-Kronecker quiver,
  and consider dimension vector~$\tuple{d}=(2,3)$.
  The stability parameter is necessarily the canonical stability parameter~$\theta=\theta_\can=(3,-2)$.
\end{example}
We have~$\Rep(Q,\tuple{d})\cong\Mat_{3\times 2}(\field)^3\cong\mathbb{A}^{18}$,
with a group action of~$\GL_2\times\GL_3$.
There are~8 Harder--Narasimhan strata,
the (semi)stable one and~7 unstable strata.
In \cref{table:harder-narasimhan-strata} we have enumerated the properties of the unstable strata.
Everything can be determined from the Euler pairings~$(\langle\tuple{d}^m,\tuple{d}^n\rangle)_{m,n=1,\ldots,\ell}$
and the tuple~$(\mu(\tuple{d}^m))_{m=1,\ldots,\ell}$.
The Harder--Narasimhan type~$((2,0),(0,3))$ is the origin in~$\mathbb{A}^{18}$,
corresponding to the direct sum~$\mathrm{S}_1^{\oplus2}\oplus\mathrm{S}_2^{\oplus3}$ of simple representations.

In the inequality \eqref{equation:teleman-inequality}
it suffices to consider the case~$m=1$ and~$n=\ell$,
because the entries of the tuple~$(k_m)_{m=1,\ldots,\ell}$ are strictly decreasing by definition,
thus the largest difference is realized by~$k_1-k_\ell$.
Therefore, in \cref{table:harder-narasimhan-strata} only this difference is listed.
We can immediately see that \eqref{equation:teleman-inequality} holds.

\begin{table}
  \centering
  \begin{tabular}{lcccccc}
    \toprule
    \multicolumn{1}{c}{$\tuple{d}^*$} & $\operatorname{codim}$ & $(\mu(\tuple{d}^m))_{m=1,\ldots,\ell}$ & $C$ & $(k_m)_{m=1,\ldots,\ell}$ & $k_1-k_\ell$ & $\eta_\lambda$ \\
    \midrule
    $((1,1),(1,2))$                   & 3                      & $(1/2,-1/3)$                           & 6   & $(3,-2)$                  & 5            & 15 \\
    $((2,2),(0,1))$                   & 4                      & $(1/2,-2)$                             & 2   & $(1,-4)$                  & 5            & 20 \\
    $((2,1),(0,2))$                   & 10                     & $(4/3,-2)$                             & 3   & $(4,-6)$                  & 10           & 100 \\
    $((1,0),(1,3))$                   & 8                      & $(3,-3/4)$                             & 4   & $(12,-3)$                 & 15           & 120 \\
    $((1,0),(1,2),(0,1))$             & 9                      & $(3,-1/3,-2)$                          & 3   & $(9,-1,-6)$               & 15           & 100 \\
    $((1,0),(1,1),(0,2))$             & 12                     & $(3,1/2,-2)$                           & 2   & $(6,1,-4)$                & 10           & 90 \\
    $((2,0),(0,3))$                   & 18                     & $(3,-2)$                               & 1   & $(3,-2)$                  & 5            & 90 \\
    \bottomrule
  \end{tabular}
  \caption{Harder--Narasimhan strata and their properties for 3-Kronecker quiver and $\tuple{d}=(2,3)$.}
  \label{table:harder-narasimhan-strata}
\end{table}

\subsection{Rigidity}
\label{subsection:rigidity}
From now on we will assume that~$Q$ is acyclic.
In \cref{section:moduli} we recalled the construction of the universal bundle~$\mathcal{U}=\bigoplus_{i\in Q_0}\mathcal{U}_i$
on~$\modulispace[\theta\stable]{Q,\tuple{d}}$.
Using the summands~$\mathcal{U}_i$ we obtain an exact sequence of vector bundles,
by combining~\cite[Proposition~3.3 and Proposition~3.7]{2307.01711v2}:
\begin{equation}
  \label{equation:4-term-exact-sequence}
  0
  \to\mathcal{O}_{\modulispace[\theta\stable]{Q,\tuple{d}}}
  \to\bigoplus_{i \in Q_0} \mathcal{U}^\vee_i \otimes \mathcal{U}_i
  \to\bigoplus_{a \in Q_1} \mathcal{U}^\vee_{\source(a)} \otimes \mathcal{U}_{\target(a)}
  \to\tangent_{\modulispace[\theta\stable]{Q,\tuple{d}}}
  \to0.
\end{equation}
See also~\cite[Section 4.1]{MR4352662} for a more direct (but less detailed) construction.

To compute the higher cohomology of the tangent bundle for the proof of \cref{corollary:rigidity},
we can split the sequence in two short exact sequences.
The cohomology of the middle terms in~\eqref{equation:4-term-exact-sequence}
is the subject of \cref{theorem:cohomology-vanishing}.
For the first term we have the following.

\begin{proposition}
  Let~$Q$, $\tuple{d}$ and~$\theta$ be as in \cref{corollary:rigidity},
  where it is possible to omit the condition that~$\tuple{d}$ is strongly amply~$\theta$-stable.
  Then
  \begin{equation}
    \label{equation:higher-cohomology-structure-sheaf-vanishing}
    \HH^k(\modulispace[\theta\stable]{Q,\tuple{d}},\mathcal{O}_{\modulispace[\theta\stable]{Q,\tuple{d}}}) = 0
  \end{equation}
  for all $k \geq 1$.
\end{proposition}

\begin{proof}
  Because~$\tuple{d}$ is chosen to be $\theta$-coprime,
  $\tuple{d}$ is a Schur root for the quiver~$Q$.
  As~$\gcd(\tuple{d}) = 1$
  we have that~$\modulispace[\theta\stable]{Q,\tuple{d}}$ is a rational variety~\cite[Theorem 6.4]{MR1914089}.
  It is also smooth and projective, as discussed in \cref{section:moduli}.
  We thus obtain the vanishing in~\eqref{equation:higher-cohomology-structure-sheaf-vanishing}
  by the birational invariance of these cohomology groups.
\end{proof}

We thus arrive at the following (short) proof of the rigidity of quiver moduli.

\begin{proof}[Proof of \cref{corollary:rigidity}]
  The higher cohomology of the first term in~\eqref{equation:4-term-exact-sequence} vanishes.
  By \cref{theorem:cohomology-vanishing} the higher cohomology of the second and third term in~\eqref{equation:4-term-exact-sequence} vanishes,
  if we in addition assume \cref{definition:strong-ample-stability}.
  This allows us to conclude.
\end{proof}

The next example shows that strong ample stability
is a strictly stronger notion than ample stability, thus justifying its name.

\begin{example}
  \label{example:strict-ample-stability}
  Let $Q$ be the 3-vertex quiver
  \begin{equation}
    \label{equation:3-vertex-quiver-1}
    Q\colon
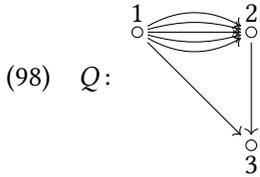
    \begin{tikzpicture}[baseline = -20pt, node distance = 1.5cm]
      \node (1)                {};
      \node (2) [right of = 1] {};
      \node (3) [below of = 2] {};
      \draw (1) circle (2pt) node[above] {1};
      \draw (2) circle (2pt) node[above] {2};
      \draw (3) circle (2pt) node[below] {3};
      \draw[->]                  (1) edge (3);
      \draw[->]                  (2) edge (3);
      \draw[->, bend left = 30]  (1) edge (2);
      \draw[->, bend left = 15]  (1) edge (2);
      \draw[->]                  (1) edge (2);
      \draw[->, bend right = 15] (1) edge (2);
      \draw[->, bend right = 30] (1) edge (2);
    \end{tikzpicture}
  \end{equation}
  Let~$\tuple{d} = (4,1,4)$,
  and let $\theta = \theta_\can = (9,-16,-5)$ be the canonical stability parameter.
  The dimension vector $\tuple{d}$ is $\theta$-coprime,
  and the stable locus $\repspace[\theta\stable]{Q,\tuple{d}}$ is non-empty.
  The representation variety admits a Harder--Narasimhan stratification with 41~strata:
  40~unstable strata, and the dense stratum of~$\theta$\dash stable representations.

  Via the formula of~\cite[Proposition 3.4]{MR1974891},
  we can verify that the codimension of the unstable locus is~$2$, therefore ample stability holds.
  On the other hand, strong ample stability is not satisfied:
  the dimension vector $\tuple{e} = (3,1,2)$ gives the inequality~$\frac{1}{6} = \mu(\tuple{e}) > \mu(\tuple{d-e}) = \frac{-1}{3}$,
  but we have $\langle \tuple{e}, \tuple{d} - \tuple{e}\rangle = -1$.
  The inequality in~\eqref{equation:teleman-inequality} can be shown to hold directly in this case,
  thus the resulting moduli space is rigid, even in the absence of strong ample stability.
\end{example}

We show now an example of a moduli space that is rigid
even though the inequality~\eqref{equation:teleman-inequality} is not satisfied.
\begin{example}
  \label{example:no-teleman-still-rigid}
  Let $Q$ be the 3-vertex quiver
  \begin{equation}
    \label{equation:3-vertex-quiver-2}
    Q\colon
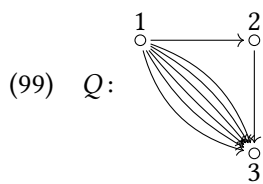
    \begin{tikzpicture}[baseline = -20pt, node distance = 1.5cm]
      \node (1)                {};
      \node (2) [right of = 1] {};
      \node (3) [below of = 2] {};
      \draw (1) circle (2pt) node[above] {1};
      \draw (2) circle (2pt) node[above] {2};
      \draw (3) circle (2pt) node[below] {3};
      \draw[->]                  (1) edge (2);
      \draw[->, bend left = 25]  (1) edge (3);
      \draw[->, bend left = 15]  (1) edge (3);
      \draw[->, bend left = 5]   (1) edge (3);
      \draw[->, bend right = 5]  (1) edge (3);
      \draw[->, bend right = 15] (1) edge (3);
      \draw[->, bend right = 30] (1) edge (3);
      \draw[->]                  (2) edge (3);
    \end{tikzpicture}
  \end{equation}
  Let $\tuple{d} = (1,6,6)$ and let $\theta = \theta_\can = (42,5,-12)$ be the canonical stability parameter.
  We have again that~$\tuple{d}$ is~$\theta$\dash coprime, and that the $\theta$-stable locus is non-empty.

  The Harder--Narasimhan stratification of the representation variety contains 85~strata:
  84~unstable strata, plus the dense stratum of stables.
  The stratum of HN-type $((0,1,0), (1,5,6))$ has codimension~$1$,
  so $\tuple{d}$ is not amply~$\theta$-stable.
  The condition in~\eqref{equation:teleman-inequality} does not hold,
  as the same stratum of HN-type $((0,1,0), (1,5,6))$ gives~$k_1 = 5, k_2 = -\frac{5}{12}$,
  and~$\langle (0,1,0),(1,5,6) \rangle = -1$.
  However, the resulting moduli space is $\mathbb{P}^6$ by \cref{lemma:P6},
  for which rigidity easily follows from the Euler sequence.
\end{example}

\begin{lemma}
  \label{lemma:P6}
  For~$Q,\tuple{d},\theta$ as in \cref{example:no-teleman-still-rigid}
  there exists an isomorphism
  \begin{equation}
    \modulispace[\theta\stable]{Q,\tuple{d}}\cong\mathbb{P}^6.
  \end{equation}
\end{lemma}

\begin{proof}
  The data of a representation $M$ consists of
  one vector~$v$ in~$M_2\cong\field^6$,
  six vectors~$w_1,\dots,w_6$ in~$M_3\cong\field^6$,
  and one endomorphism of~$\field^6$ that we denote by~$A \colon \field^6 \to \field^6$.
  Let us consider the conditions on this data for $M$ to be $\theta$-stable.
  If $A$ has a nontrivial kernel, then $M$ admits a subrepresentation of dimension vector $(0,\dim_\field\ker A,0)$ or $(1,\dim_\field\ker A,0)$,
  and both have positive slope.
  Thus $A$ must be injective for~$M$ to be~$\theta$-stable.
  If $Av,w_1,\dots,w_6$ together do not span $M_3$,
  then there is a subrepresentation of dimension vector $(1,6,5)$ which has positive slope.
  Therefore, the $6 \times 7$-matrix
  \begin{equation}
    F_{M} \colonequals\left( Av~|~w_1~|~\dots~|~w_6 \right)
  \end{equation}
  must have rank $6$.
  Using this description we can see that
  the moduli space can be identified to $\mathbb{P}^6$.
  To do so, we send the data of a semistable representation $M$ to
  the kernel of the linear map $F_M \colon \field^7 \to \field^6$,
  which is a line in~$\field^7$.
\end{proof}

\subsection{Height-zero moduli spaces}
\label{subsection:height-zero}
Let us recall some aspects of the moduli theory of semistable sheaves on~$\mathbb{P}^2$.
We will denote
the moduli space of Gieseker-semistable sheaves with given rank and first (resp.~second) Chern class
by~$\moduli_{\mathbb{P}^2}(r,\mathrm{c}_1,\mathrm{c}_2)$.
In~\cite{MR0916199} a function~$\delta\colon\mathbb{Q}\to\mathbb{Q}$ is constructed,
for which it is shown that~$\moduli_{\mathbb{P}^2}(r,\mathrm{c}_1,\mathrm{c}_2)$ has strictly positive dimension
if and only if
\begin{equation}
  \delta\left( \frac{\mathrm{c}_1}{r} \right)
  \leq
  \frac{1}{r}\left( \mathrm{c}_2 - \left( 1-\frac{1}{r} \right)\frac{\mathrm{c}_1^2}{2} \right).
\end{equation}
When it is an \emph{equality},
the moduli space is said to be of \emph{height zero}.

It is also shown that for moduli spaces of strictly positive dimension
there exists a unique exceptional vector bundle.
Writing~$\mu\colonequals\frac{\mathrm{c}_1}{r}$ for the slope of the sheaves parametrized by~$\moduli_{\mathbb{P}^2}(r,\mathrm{c}_1,\mathrm{c}_2)$,
we will denote this associated exceptional vector bundle by~$E_\mu$.

As explained in~\cite{MR0916199},
moduli spaces of height zero have favourable properties.
The one which is relevant to us is the following identification~\cite[Th\'eor\`eme~2]{MR0916199}.

\begin{theorem}[Drezet]
  \label{theorem:identification}
  There exists a natural isomorphism
  \begin{equation}
    \moduli_{\mathbb{P}^2}(r,\mathrm{c}_1,\mathrm{c}_2)
    \cong
    \modulispace[\theta_\can\semistable]{\mathrm{K}_{3r_\mu},(m,n)}
  \end{equation}
  where~$r_\mu$ is the rank of the of the associated exceptional vector bundle~$E_\mu$,
  $\mathrm{K}_{3r_\mu}$ is the Kronecker quiver with~$3r_\mu$ arrows,
  and the dimension vector~$(m,n)$ is determined as in~\cite[\S IV.2]{MR0916199}.
\end{theorem}

The isomorphism restricts to an isomorphism of the stable loci,
i.e.,
when the dimension vector~$(m,n)$ is not coprime.
In what follows,
we only consider invariants~$(r,\mathrm{c}_1,\mathrm{c}_2)$
such that every semistable sheaf is stable,
and thus~$(m,n)$ will be coprime.

For Kronecker quivers,~\cite[Proposition 6.2]{MR3683503} shows that strong ample stability holds.
The cited statement only mentions ample stability,
but the proof shows strong ample stability,
as introduced earlier in~\cite[Proposition 5.1]{MR3683503}.
We thus have the following lemma.

\begin{lemma}
  \label{lemma:rigidity-kronecker-quivers}
  Let $Q$ be a Kronecker quiver,
  and let~$\tuple{d}$ be a dimension vector
  which is~$\theta_\can$-coprime (and thus indivisible).
  Then~$\modulispace[\theta\semistable]{Q,\tuple{d}}$ is rigid.
\end{lemma}

To prove \cref{corollary:height-zero},
we wish to obtain a contradiction on the assumption that~\eqref{equation:height-zero-functor} is fully faithful.
The combination of~\cite[Proposition~29, Remark~30]{MR3950704} states the following
for fully faithful functors towards moduli spaces of sheaves on surfaces.
\begin{theorem}[Belmans--Fu--Raedschelders]
  \label{theorem:fully-faithful-hochschild-cohomology}
  Let~$S$ be a smooth projective surface
  such that~$\mathcal{O}_S$ is expectional,
  i.e.,~$\HH^1(S,\mathcal{O}_S)=\HH^2(S,\mathcal{O}_S)=0$.
  Let~$\moduli_S(r,\mathrm{c}_1,\mathrm{c}_2)$ be a smooth projective moduli space
  of stable sheaves on~$S$ with prescribed invariants.
  Let~$\mathcal{E}$ be the universal sheaf.
  Assume that the Fourier--Mukai functor
  \begin{equation}
    \Phi_{\mathcal{E}}\colon\derived^\bounded(S)\to\derived^\bounded(\moduli_S(r,\mathrm{c}_1,\mathrm{c}_2))
  \end{equation}
  is fully faithful.
  Then
  \begin{equation}
    \HH^i(\moduli_S(r,\mathrm{c}_1,\mathrm{c}_2),\tangent_{\moduli_S(r,\mathrm{c}_1,\mathrm{c}_2)})
    \cong
    \hochschild^{i+1}(S).
  \end{equation}
\end{theorem}

\begin{proof}[Proof of \cref{corollary:height-zero}]
  For~$\mathbb{P}^2$,
  the Hochschild--Kostant--Rosenberg decomposition gives us the isomorphism~%
  $\hochschild^2(\mathbb{P}^2)\cong\HH^0(\mathbb{P}^2,\omega_{\mathbb{P}^2}^\vee)\cong\field^{10}$.
  By the identification from \cref{theorem:identification}
  and the assumption that the moduli spaces are smooth
  we have that
  \begin{equation}
    \label{equation:first-cohomology-tangent-bundle-identification}
    \HH^1(\moduli_{\mathbb{P}^2}(r,\mathrm{c}_1,\mathrm{c}_2),\tangent_{\moduli_{\mathbb{P}^2}(r,\mathrm{c}_1,\mathrm{c}_2)})
    \cong
    \HH^1(\modulispace[\theta_\can]{\mathrm{K}_{3r_\mu},(m,n)},\tangent_{\modulispace[\theta_\can]{\mathrm{K}_{3r_\mu},(m,n)}}).
  \end{equation}
  But by \cref{corollary:rigidity},
  which we can apply because of \cref{lemma:rigidity-kronecker-quivers},
  the cohomology in~\eqref{equation:first-cohomology-tangent-bundle-identification} vanishes,
  which contradicts \cref{theorem:fully-faithful-hochschild-cohomology}.
\end{proof}

The following remark gives an alternative method to prove
the functor cannot be fully faithful,
using techniques introduced already by Drezet.
\begin{remark}
  \label{remark:associated-exceptional-bundle}
  Consider~$\moduli_{\mathbb{P}^2}(r,\mathrm{c}_1,\mathrm{c}_2)$ of height zero.
  As explained in~\cite[\S5.1]{MR0916199},
  there exists a unique associated exceptional vector bundle~$E$
  for which
  \begin{equation}
    \label{equation:inequalities}
    -3<\frac{\mathrm{c}_1H}{r}+\frac{\mathrm{c}_1(E)H}{\rk E}\leq 0.
  \end{equation}
  As mentioned in loc.~cit.,
  this implies that
  \begin{equation}
    \chi(\mathbb{P}^2,E^\vee\otimes V)=\operatorname{ht}(\moduli_{\mathbb{P}^2}(r,\mathrm{c}_1,\mathrm{c}_2))=0
  \end{equation}
  for all~$[V]\in\moduli_{\mathbb{P}^2}(r,\mathrm{c}_1,\mathrm{c}_2)$.
  But then we can prove that~$\Phi_{\mathcal{U}}(E^\vee)=0$,
  contradicting fully faithfulness.
  Namely, fiberwise we see for the transform of~$E^\vee$ that
  \begin{equation}
    \HH^\bullet(\mathbb{P}^2,(p^*E^\vee\otimes\mathcal{U})_{\mathbb{P}^2\times[V]})
    \cong
    \HH^\bullet(\mathbb{P}^2,E^\vee\otimes V)
    =
    0,
  \end{equation}
  as both~$\Hom(E,V)$ and~$\Ext^2(E,V)=\Hom(V,E\otimes\mathcal{O}_{\mathbb{P}^2}(3))^\vee$ vanish,
  by the inequalities~\eqref{equation:inequalities}.
\end{remark}

\printbibliography

\emph{Pieter Belmans}, \url{pieter.belmans@uni.lu} \\
Department of Mathematics, Universit\'e de Luxembourg, 6, avenue de la Fonte, L-4364 Esch-sur-Alzette, Luxembourg

\emph{Ana-Maria Brecan}, \url{anabrecan@gmail.com}

\emph{Gianni Petrella}, \url{gianni.petrella@uni.lu} \\
Department of Mathematics, Universit\'e de Luxembourg, 6, avenue de la Fonte, L-4364 Esch-sur-Alzette, Luxembourg

\emph{Hans Franzen}, \url{hans.franzen@math.upb.de} \\
Institute of Mathematics, Paderborn University, Warburger Stra\ss e 100, 33098 Paderborn, Germany

\emph{Markus Reineke}, \url{markus.reineke@rub.de} \\
Fakultat f\"ur Mathematik, Ruhr-Universit\"at Bochum, Universit\"atsstra\ss e 150, 44780 Bochum, Germany

\end{document}